\documentclass[a4paper,10pt]{amsart}
\usepackage[english]{babel}
\usepackage{a4wide}
\usepackage[utf8]{inputenc}
\usepackage{csquotes}

\usepackage{graphicx}
\usepackage{subcaption}
\usepackage[colorlinks=true, allcolors=blue]{hyperref}
\usepackage{amsmath,graphicx,amssymb}
\usepackage{amssymb,amsmath,amscd,amsfonts,mathrsfs,amsthm,bbm,dsfont,stmaryrd }
\usepackage{color}
\usepackage{stackengine}
\usepackage{enumitem}
\usepackage{ulem}

\newtheorem{theorem}{Theorem}
\newtheorem{lemma}{Lemma}
\numberwithin{lemma}{section}
\newtheorem*{theorem*}{Theorem}
\newtheorem{corollary}[theorem]{Corollary}
\newtheorem{proposition}[lemma]{Proposition}
\newtheorem{definition}[lemma]{Definition}
\newtheorem{remark}[lemma]{Remark}
\newtheorem{example}[lemma]{Example}

\newtheorem*{condition*}{Condition}

\newcommand{\RR}{\mathbb{R}} 
\newcommand{\EE}{\mathbb{E}} 

\newcommand{\indep}{\perp \!\!\! \perp}
\newcommand{\laweq}{\stackrel{\mathcal{L}}{=}}
\newcommand{\cconverge}{\stackrel{c}{\xrightarrow[n \to +\infty]{}}}

\newcommand{\sperp}{\raisebox{-0.7pt}{\scalebox{0.7}[0.35]{$\perp$}}}
\newcommand{\barp}[1]{\overset{\sperp}{#1}}
\newcommand{\barpp}[2]{\barp{#1}^{\raisebox{-2pt}{$\scriptstyle #2$}}}

\newcommand{\bs}{\boldsymbol}

\newcommand{\balpha}{\boldsymbol\alpha} 
\newcommand{\cI}{\mathcal I} 
\newcommand{\cF}{\mathcal F} 
\newcommand{\cL}{\mathcal L} 
\newcommand{\cN}{\mathcal N} 
 
\newcommand{\bq}{\boldsymbol q} 
\newcommand{\bu}{\boldsymbol u} 
\newcommand{\varnew}{\barp{\sigma}^{\raisebox{-2pt}{$\scriptstyle 2$}}}
\newcommand{\stdnew}{\overset{\sperp}{\sigma}}

\newcommand{\eqlaw}{\stackrel{\mathcal L}{=}} 
\newcommand{\cond}[1]{\stackrel{\mathcal L}{=}_{| #1}}

\DeclareMathOperator{\LCP}{LCP}
\DeclareMathOperator{\Elliptic}{Elliptic}

\DeclareMathOperator{\colspan}{span}

\makeatletter
\renewcommand\paragraph{\@startsection{paragraph}{4}{\z@}%
  {-3.25ex\@plus -1ex \@minus -.2ex}%
  {1.5ex \@plus .2ex}%
  {\normalfont\normalsize\bfseries}}
\makeatother

\title[Elliptic Approximate Message Passing]{Elliptic Approximate Message 
 Passing and an application to theoretical ecology}
\author[Gueddari et al.]{Mohammed-Younes Gueddari, Walid Hachem, Jamal Najim}

\date{\today\\
CNRS, Laboratoire d'informatique Gaspard Monge (LIGM / UMR 8049),
Universit\'e Gustave Eiffel, France} 

\begin{document}

\maketitle

\begin{abstract}

Approximate Message Passing (AMP) algorithms
have recently gathered significant attention across disciplines such as statistical physics, machine learning, and communication systems. This study aims to extend AMP algorithms to non-symmetric (elliptic) matrices, motivated by analyzing equilibrium properties in ecological systems featuring elliptic interaction matrices.

In this article, we provide the general form of an AMP algorithm associated to a random elliptic matrix, the main change lying in a modification of the corrective (Onsager) term. In order to establish the statistical properties of this algorithm, we use and prove a generalized form of Bolthausen conditioning argument, pivotal to proceed by a Gaussian-based induction.

We finally address the initial motivating question from theoretical ecology. Large foodwebs are often described by Lotka-Volterra systems of coupled differential equations, where the interaction matrix is elliptic random. In this context, we design an AMP algorithm to analyze the statistical properties of the equilibrium point in a high-dimensional regime. We rigorously recover the results established by Bunin \cite{Bunin} and Galla \cite{Galla_2018} who used techniques from theoretical physics, and extend them with the help of {\it Propagation of chaos} type arguments.

\end{abstract}

\section{Introduction}

Approximate Message Passing (AMP) is a class of versatile and configurable iterative algorithms. The output is a sequence of high-dimensional $\mathbb{R}^n$-valued random vectors $(\bs{u}^k)_{k\ge 1}$ based on $n\times n$ (usually symmetric) random matrices, see \eqref{eq:AMP_symmetric}. An important feature of AMP is a precise description of the $(\bs{u}^k)$'s statistical properties as $n$ goes to infinity, mainly via the so-called Density Evolution (DE) equations \eqref{eq:DE}. 

Initially used in statistics for solving compressed sensing and sparse signal recovery problems \cite{Donoho_2009,bayati2011dynamics}, the AMP algorithms have found numerous applications in the fields of high-dimensional estimation \cite{desh-abb-mon-17,lel-mio-ptrf19}, communication theory 
\cite{bar-krz-it17,rush-gre-ven-17},  or 
statistical physics \cite{mon-siam-19},  
and have undergone extensive developments that have widened their spectrum of applications.

The goal of this article is to extend the AMP procedure to a non-Hermitian setting and to consider large elliptic random matrices. This new setting is mainly motivated by the study of equilibria in large random Lotka-Volterra systems of differential equations, a popular model in theoretical ecology. Interestingly, modifying the matrix nature in the AMP algorithm changes the iteration equation (modification of the {\it Onsager term}), necessitates extra mathematical developments but does not modify the DE equations.


\subsection*{A primer on Approximate Message Passing} 
Recall the definition of the Gaussian Orthogonal Ensemble (GOE), a $n\times n$ random matrix with representation $(X + X^\top) / \sqrt{2}$ where $X\in \mathbb{R}^{n\times n}$ has $n^2$ independent $\cN(0,1)$ elements. Consider the following AMP iterative algorithm :
\begin{equation}
\label{eq:AMP_symmetric}
\begin{cases}
    \bs{u}^1 = A_n\, h_{0}\left(\bs{u}^0, \bs{b}\right) \\
    \bs{u}^{k+1} = A_n\, h_{k}\left(\bs{u}^k, \bs{b}\right) 
- \left\langle \partial_1 h_k\left(\bs{u}^k,\bs{b}\right)\right\rangle_n
 h_{k-1}\left(\bs{u}^{k-1}, \bs{b}\right), \quad \text{for} \ k\geq 1
\end{cases}
\end{equation}
where $\bs{u}^{0}\in\mathbb{R}^n$ is an initialization vector that can be either random or deterministic, $\bs{b} = (b_i)_{i\in[n]}$ is a parameter vector and $\bs{u}^{k} = ( u^k_i)_{i\in[n]} \in \RR^n$ is the iterate at step $k$.
Matrix $A_n$ is $n\times n$ random such that $\sqrt{n} A_n$ is drawn from GOE. Functions $h_{k} : \RR^2 \to \RR $ ($k\ge 0$ with the convention that $h_{-1}=0$) are the so-called activation functions and applied componentwise to vectors $\bs{u}^k=(u^k_i)$ and $\bs{b}=(b_i)$: 
$$
h_{k}(\bs{u}^k, \bs{b}) = \left( h_{k}({u}^k_i, {b}_i)\right) \in \RR^n\, .
$$
These functions are assumed to be differentiable with respect to the first parameter. We denote the 
partial derivative $\frac{\partial h}{\partial u} (u,b)$ by $\partial_1 h(u,b)$, and introduce the notation:
$$
\left\langle \partial_1 h_k\left(\bs{u}^k,\bs{b}\right)\right\rangle_n = \frac 1n  \sum_{i\in[n]} \partial_1 h_k\left({u}^k_i,{b}_i\right)\in \mathbb{R}\, .
$$
\subsubsection*{The Onsager term} In the AMP literature, the so-called {\bf Onsager term} defined by
$$- \left\langle \partial_1 h_k\left(\bs{u}^k,\bs{b}\right)\right\rangle_n
 h_{k-1}\left(\bs{u}^{k-1}, \bs{b}\right)$$
plays a pivotal role to describe the (high dimensional) statistical properties of $(\bs{u}^0,\cdots, \bs{u}^k)$. It 
is designed to asymptotically remove (as $n$ goes to infinity) the non-Gaussian component from $A_n \, h_k\left(\bs{u}^k,\bs{b}\right)$. 

\subsubsection*{The joint empirical distribution} The joint empirical distribution of $(\bs{u}^1,\cdots, \bs{u}^k)$ is defined as
$$
\mu^{\bs{u}^1,\cdots, \bs{u}^k} = \frac{1}{n}\sum_{i\in [n]} \delta_{(u_i^1, \cdots, u_i^k)}\,,
$$
where $\delta_{(u_1,\cdots u_k)}$ is the Dirac distribution at point $(u_1,\cdots,u_k)$. The techniques developed within the scope of AMP enable to describe the limit of this joint empirical distribution as $n\to \infty$, which turns out to be the distribution of a centered $k$-dimensional Gaussian vector whose covariance matrix is defined recursively by the {\bf Density Evolution equations}.

\subsubsection*{Density Evolution (DE) equations} 
These equations recursively characterize a family of covariance matrices $R^{k} \in \RR^{k\times k}$ with $k\ge 1$ defined as follows. 
\begin{itemize}
\item[-] Initialization: Let $(\bar{u},\bar{b})$ be a random vector in $\RR^2$. Set 
\begin{equation}\label{eq:DE-init}
R^1 = \mathbb{E}\left[h_0^2\left(\bar{u}, \bar{b}\right)\right]
\end{equation}
and let $Z_1\sim {\mathcal N}(0,R^1)$ be a random variable independent from $(\bar{u},\bar{b})$. 
\item[-] Recursion: Suppose that $R^{k-1}$ is given. Let $(Z_1,\cdots,Z_{k-1})\sim {\mathcal N}_{k-1}(0,R^{k-1})$ be a vector independent from $(\bar{u},\bar{b})$. Let $R^k=(R^k_{ij})$ be a $k\times k$ matrix defined by 
\begin{equation}
        \label{eq:DE}
            R^k_{ij} = \begin{cases}
   R^1 & \mbox{ if } i=1, j=1\\
       \mathbb{E}\left[h_{i-1}\left(Z_{i-1}, \bar{b}\right)h_{j-1}\left(Z_{j-1}, \bar{b}\right)\right] & \mbox{ if } i \geq 1 , j\geq 1 \\
   \end{cases}\,
\end{equation}
with the convention that $Z_0 \triangleq \bar{u}$. 
\end{itemize}
Notice that $R^{k-1}$ represents the upper-left corner of matrix $R^k$. Thus we can also define an operator $R$ as follows which will encode all matrices $R^k$:
\begin{equation}
\label{eq:DEOperator}
    \begin{array}{cccccc}
R & : & \mathbb{N}^*\times\mathbb{N}^* & \to & \mathbb{R}_+ &\\
 & & (i,j) & \mapsto & R_{ij}^k & \text{where } k\geq i,j,\\
\end{array}
\end{equation}
We will simply denote $R(i,j)$ by $R_{ij}$ for all $(i,j)\in \mathbb{N}^*\times\mathbb{N}^*.$

We now (informally) state the main result of the AMP for a GOE matrix: 
\begin{theorem*}[\cite{bayati2011dynamics,feng2021unifying}]
\label{th:AMP_conv}
    Let $A_n$ be a $n\times n$ matrix such that $\sqrt{n}A_n$ is drawn from the GOE. Let $\bs{u}^0, \bs{b}\in \mathbb{R}^n$ independent from $A_n$, and $(\bs{u}^\ell)_{1\le \ell\le k}$ be defined by \eqref{eq:AMP_symmetric}. Suppose that $\mu^{\bs{u}^0,\bs{b}} \xrightarrow[n\to \infty]{} {\mathcal L}(\bar{u},\bar{b})$  
    and let $R^k$ be defined by the DE equations \eqref{eq:DE-init}-\eqref{eq:DE}. Then
$$
\mu^{\bs{u}^1,\cdots, \bs{u}^k} \xrightarrow[n\to\infty]{}  {\mathcal L}(Z_1,\cdots, Z_k) \qquad \textrm{where}\quad (Z_1,\cdots, Z_k)\sim {\mathcal N}_k(\bs{0}, R^k)\, .
$$
\end{theorem*}
The nature of the convergence of measures $\mu^{\bs{u}^0,\bs{b}}$ and $\mu^{\bs{u}^1,\cdots, \bs{u}^k}$ will be specified for the main theorem of the article.

\subsubsection*{About the literature} Numerous studies have extended AMP algorithms to more complex and general scenarios \cite{javanmard2013state,Bayati_2015}. For instance, Vector AMP \cite{rangan2018vector} broadens AMP to handle vector observations, considering $\bs{u}^{k}$ as a matrix rather than a vector. This adaptation is suited for multi-channel and multi-dimensional signal processing tasks. Fan \cite{fan2021approximate} generalizes AMP algorithms to encompass a wide range of matrices, particularly rotationally invariant ones, Dudeja et al. \cite{dudeja2023universality} explore universality properties of AMP, etc. 

Another important generalization known as ``Asymmetric AMP'' is discussed in \cite[section 2.2]{feng2021unifying}. Given a rectangular matrix $A$ with i.i.d. Gaussian entries, the form of this AMP algorithm can be summarized as the following:
\begin{equation*}
    \begin{cases}
        \bs{v}^k &= A \ g_k\left(\bs{u}^k,\alpha\right) - \beta_k h_{k-1}\left(\bs{v}^{k-1},\bs{b}\right)\\
        \bs{u}^{k+1} &= A^\top h_k\left(\bs{v}^k,\bs{b}\right)-b_k g_k\left(\bs{u}^k,\alpha\right)
    \end{cases},
\end{equation*}
which is a two-step algorithm that involves both matrices $A$ and $A^\top$. 

None of these extensions cover our model of interest.
\subsection*{Non-Hermitian AMP} In the sequel we will consider an AMP algorithm based on an elliptic matrix instead of a GOE matrix. For the sake of our application, we shall also enable multiple parameter vectors $(\bs{b}^1,\cdots, \bs{b}^p)$ instead of a single one ($p$ fixed). These vectors will be stacked into a $n\times p$ matrix
\begin{equation}
\label{eq:multipleVector}
    B=(\bs{b}^1,\cdots,\bs{b}^p)
\end{equation}

and for $h:\mathbb{R}^{k+1}\to \mathbb{R}$, $\bs{u}\in \mathbb{R}^n$ , $h(\bs{u}, B)$ will denote the vector
$$
h(\bs{u},B) = \left( \,h(u_i,B_{i1},\cdots, B_{ip})\,\right)_{i\in [n]}\, .
$$

\begin{definition}[Gaussian elliptic matrix model]
\label{def:elliptic-model} 
 A random matrix $M_n=(M_{ij}) \in \RR^{n\times n}$ is said to follow the Gaussian elliptic 
 distribution with parameter $\rho\in [-1,1]$ if the entries' distributions are given by
 $$
 M_{ii}\sim \cN(0, 1 + \rho)\quad \textrm{for}\ i\in [n]\qquad \textrm{and} \qquad 
 \begin{bmatrix} M_{ij}\\ M_{ji} \end{bmatrix}\sim \cN_2\left( \bs{0}, \begin{bmatrix} 1 & \rho\\ \rho  & 1 \end{bmatrix}\right)\quad \textrm{for}\  i<j\, .
 $$
 Moreover, all the elements of the following set are independent: 
 $$
 \Big\{M_{ii} \,, \ i\in [n]\Big\}\cup 
  \Big\{(M_{ij}, M_{ji}) \,, \ i,j \in [n] \,, i < j\Big\}\, .
 $$
\end{definition}
We write $M_n\sim \Elliptic(n,\rho)$ for such matrices. We will also call a
\textit{normalized (Gaussian) elliptic matrix} $A_n \in \RR^{n\times n}$ a matrix that
verifies $\sqrt{n} A_n \sim \Elliptic(n,\rho)$.

Notice that a normalized elliptic matrix $A_n$ is no longer symmetric. Its spectral distribution has been thoroughly studied \cite{girko1986elliptic,naumov2012elliptic,10.1214/EJP.v19-3057} and it is well-known that it almost surely converges as $n\to \infty$ to the uniform law over the compact set bounded by the ellipse defined for $|\rho| < 1$ by
\[
\mathscr E_\rho = \left\{ (x,y) \in \RR^2 \ : \ 
  \frac{x^2}{(1+\rho)^2} + \frac{y^2}{(1-\rho)^2} \leq 1 \right\}\,,
\] 
hence the name. This model interpolates from antisymmetric matrices ($\rho=-1$) to GOE matrices ($\rho=1$) with the important special case of matrices with i.i.d. Gaussian entries ($\rho=0$).  
\subsubsection*{Elliptic AMP}
For a normalized elliptic matrix $A_n$, the AMP algorithm takes the form:
\begin{equation}
\label{eq:AMP_elliptic}
\bs{u}^{k+1} = A_n \ h_{k}\left(\bs{u}^k, B\right) - \rho 
\left\langle \partial_1 h_k\left(\bs{u}^k,B\right)\right\rangle_n
h_{k-1}\left(\bs{u}^{k-1}, B\right)\,,
\end{equation}
and we shall prove in Theorem \ref{th:main} the counterpart of Theorem \ref{th:AMP_conv} with the same DE equations, vector $\bs{b}$ being replaced by matrix $B$.

Notice however the modified Onsager term, multiplied by the correlation coefficient $\rho$. If $\rho=0$, matrix $\sqrt{n}A$ has i.i.d. $\cN(0,1)$ entries and the Onsager term vanishes. For $\rho=1$ we recover the GOE model.  


\subsubsection*{Bolthausen's conditioning argument} In the framework of the GOE, the challenge for rigorously proving Theorem \ref{th:AMP_conv} by induction lies in an essential ingredient often called  Bolthausen's conditioning argument \cite{bolthausen2014iterative}. This technique provides a convenient representation of the conditional distribution of the new iterate $\bs{u}^{k+1}$ given the past $\left(\bs{u}^1,\cdots,\bs{u}^k\right)$. In the course of the proof, we establish a generalized version of Bolthausen's technique suited to handle elliptic matrix models.

\subsection*{Application to theoretical ecology}

An important challenge in theoretical ecology is to build and analyze mathematical models to describe trophic networks and food-webs \cite{leibold2004metacommunity} of large dimension. In this regard, Lotka-Volterra (LV) systems of coupled differential equations is a popular model to describe the evolution of the various species' abundances in a food-web. 
The LV system of equations is written:
\begin{equation}\label{eq:LV-system}
\frac{d\bs{x}}{dt} (t) = \bs{x}(t) \odot 
  \left( \bs{r} - \left(I_n - \Sigma_n \right) \bs{x}(t) \right), 
  \quad 
 \bs{x}(0) \in (0,\infty)^n ,  
\end{equation}
where $\odot$ stands for the Hadamard product, $\bs{x}(t)\in \mathbb{R}^n$ is the vector of abundances of the $n$ species at time $t$, $\bs{r}\in \mathbb{R}^n$ is the vector of intrinsic growth rates of the species and $\Sigma_n=(\Sigma_{ij})$ is the $n\times n$ interaction matrix, $\Sigma_{ij}$ representing the effect of species $j$ on the growth of species $i$. 

In large dimension, a key feature of LV systems is the use of random matrices\footnote{In theoretical ecology, the use of random matrices goes back to May \cite{may1972will}.}, see for instance \cite{allesina2015stability,akjouj2022complex}, to model the interactions between the different species. This choice of a random matrix model is motivated in particular because the estimation of the real interactions is often out of reach. 

An elliptic interaction matrix $\Sigma_n$, more precisely $\kappa \sqrt{n}
\Sigma_n \sim \textrm{Elliptic}(n,\rho)$ (here $\kappa>0$ is an extra degree of
freedom), covers the case where the reciprocal interactions $\Sigma_{ij}$ and
$\Sigma_{ji}$ between species are correlated. The elliptic model
encompasses the cases of independent and equal reciprocal ecological
interactions, and is widely considered in theoretical ecology
\cite{allesina2012stability,allesina2015stability,Bunin,Galla_2018}. 

The question we shall address is the description of the statistical properties of an equilibrium $\bs{x}^\star\in \mathbb{R}^n$ to \eqref{eq:LV-system}, as $n$ goes to infinity, whenever such an equilibrium exists (sufficient conditions for the existence of a unique and stable equilibrium have been provided in \cite{clenet2022equilibrium}). More specifically, we will be interested in the number of surviving species at equilibrium, the distribution of the surviving species, the individual distribution of a species, etc. 

Based on Theorem \ref{th:main} and following the strategy developed in \cite{akjouj2023equilibria} in the context of a symmetric interaction matrix, we will design an AMP algorithm which shall capture the equilibrium's statistical properties. This question has already been addressed by Bunin \cite{Bunin} and Galla \cite{Galla_2018} who provided a full description of $\bs{x}^\star$'s statistical properties via a system of non-linear equations at a physical level of rigor. We recover their equations, cf. \eqref{eq:sys}, and provide a rigorous analysis of this system, substantially more demanding than in the symmetric case \cite{akjouj2023equilibria}.

Combining a local AMP result (see Corollary \ref{cor:partialAMP}) with
arguments from the propagation of chaos theory, we also obtain new results on
the individual distributions of species with different intrinsic growth rates.


\subsection*{Outline of the article} In Section \ref{sec:AMP}, we present the elliptic AMP algorithm and state the main corresponding results, Theorem \ref{th:main} (global AMP) and Corollary \ref{cor:partialAMP} (blockwise AMP). In Section \ref{sec:ecology}, we present an application of AMP to theoretical ecology and design a specific AMP algorithm to describe the statistical properties of an equilibrium to a large LV system, see Theorem \ref{th:main-LV}. Relying on propagation of chaos arguments, we describe the limiting behaviour of individual species' abundances in Corollary \ref{coro:chaos-prop-I} (global exchangeability assumption) and Theorem \ref{th:chaos-prop-II} (blockwise exchangeability assumption).   
Proofs of AMP results are provided in Section \ref{sec:proofs}. Proofs related to LV equilibria are provided in Section \ref{sec:proofs-ecology}.

Technical results of special interest are Lemma \ref{lemma:main-system} (description of the key equilibrium parameters via a deterministic system), Propositions \ref{prop:bolt} and \ref{prop:bolt++} (extension of Bolthausen conditioning argument to elliptic random matrices) and Proposition \ref{prop:sznitman_generalized} (chaos propagation for blockwise exchangeable vectors).

\subsection*{Main notations and definitions} For a positive integer $n$, denote $[n]=\{1,\cdots n\}$. 
For $x\in \mathbb{R}$ let $x_+=\max(x,0)$ and $x_-=-\min(x,0)$ so that $x=x_+-x_-$. 
Vectors will be denoted by lowercase bold letters $\bs{x} = (x_i)\in \RR^n$ and matrices by capital letters. For a matrix $A=(A_{ij})\in \RR^{n\times n}$, we denote by $A_{i,*}$ its $i$-th row and by $A_{*,j}$ its $j$-th column we also denote by $A_{[i],j}$ the first $i$ elements of the $j$-th column of $A$. We denote by $A^\top$ the transpose transpose matrix of $A$.

For a vector $\bs{x}\in \RR^n$ (respectively a $n\times n$ matrix $B$), $\lVert \bs{x} \rVert $ (resp. $\| B\|$) denotes its euclidean norm (resp. spectral norm) and $\langle \bs{x} \rangle_n \triangleq \frac{1}{n}\sum_{i=1}^n x_i$ the arithmetic mean of its coordinates. For two matrices $A,B$ with identical dimensions, denote by $A\odot B= (A_{ij} B_{ij})$ their Hadamard product. The notation applies for two $\mathbb{R}^n$-vectors $\bs{x}\odot \bs{y}=(x_i y_i)$.

For $f:\RR \rightarrow \RR$, $g:\RR^{p+1}\rightarrow \mathbb{R}$ and $\bs{x},\bs{y}^1,\cdots, \bs{y}^{p}\in \mathbb{R}^n$, denote by $f(\bs{x})$ and $g(\bs{x},\bs{y}^1,\cdots, \bs{y}^p)$ the $n$-dimensional vectors
$$
f(\bs{x}) = \left(f(x_i)\right)_{i\in [n]}\qquad \text{and}\qquad g(\bs{x},\bs{y}^1,\cdots, \bs{y}^p)= (g(x_i,y^1_i,\cdots, y_i^p))_{i\in [n]}\, .
$$ 
In particular, $\bs{x}_+=([x_i]_+)$. 

Denote by ${\mathcal L}(X)$ the law of a random variable $X$. The equality in law between $X$ and $Y$ will be either denoted $X \laweq Y$ or $\mathcal{L}(X) = \mathcal{L}(Y)$. Independence is denoted by $\indep$.

\subsection*{Acknowlegment} We thank all the members of the CNRS project 80-Prime-KARATE where part of this work has been initiated.

\section{AMP for random elliptic matrices}
\label{sec:AMP}

We first introduce the notions of complete convergence and Wasserstein spaces. These concepts are crucial for the precise formulation of our main theorem.

\subsection{Background} 
\subsubsection*{Complete convergence}
Given a sequence of random variables $(X_n)$, we say that $X_n$ converges completely to a constant $x$ if for any other sequence $(Y_n)_n$ such that $X_n \laweq Y_n$ for all $n$, $Y_n$ converges almost surely to $x$. We denote this mode of convergence as 
$$
X_n \cconverge x\qquad \text{or}\qquad X_n \xrightarrow[n\to\infty]{} x \quad \textrm{(completely)}\,.
$$
It is worth noticing that the complete convergence of $(X_n)_n$ to $x$ is equivalent to the condition $\sum_{n\in \mathbb{N}} \mathbb{P}(\lVert X_n - x \rVert > \varepsilon) < \infty$ for all $\varepsilon > 0$, as per Borel-Cantelli's lemma. One advantage of this convergence mode is that it is transmissible through equality in law, i.e.:

\begin{equation*}
    \text{If} \quad \begin{cases}
        X_n \cconverge x, \phantom{\Bigg|}\\
        Y_n \laweq X_n, \quad \text{for all} \ $n$
    \end{cases} \ \text{then} \quad Y_n \cconverge x.
\end{equation*}
This property is shared with convergence in probability but not with almost sure convergence.

\subsubsection*{Wasserstein spaces}
The Wasserstein space of order $r\geq 2$ denoted by $\mathcal{P}_r\left(\RR^d\right)$ is the set of probability distributions $\mu$ on $\RR^d$ with finite moments of order $r$:
\begin{equation*}
    \mathcal{P}_r\left(\RR^d\right) = \left\{\mu\in {\mathcal P}(\mathbb{R}^d)\,,\quad \int_{\RR^d} \lVert \bs{x} \rVert^r d\mu(\bs{x}) < \infty \right\}\,.
\end{equation*}
The Wasserstein distance between $\mu, \nu\in \mathcal{P}_r\left(\RR^d\right)$, denoted by $ \mathcal{W}(\mu,\nu)$, is defined as:
\begin{equation*}
    \mathcal{W}(\mu,\nu) = \inf_{\pi \in \Pi(\mu,\nu)} \left(\int_{\RR^d \times \RR^d} \lVert \bs{x}-\bs{y} \rVert^r d \pi(\bs{x},\bs{y}) \right)^{1/r},
\end{equation*}
where $\Pi(\mu,\nu)$ is the set of probability measures on $\RR^d\times \RR^d$ having $\mu$ and $\nu$ as marginals.
Given a sequence of probability measures $(\mu_n)\subset \mathcal{P}_r\left(\RR^d\right)$ we will say that this sequence converges in the Wasserstein space to $\mu\in \mathcal{P}_r\left(\RR^d\right)$ if $\mathcal{W}(\mu_n, \mu)\xrightarrow[n\to\infty]{} 0$. We write 
$$
\mu_n \stackrel{\mathcal{P}_r\left(\RR^d\right)}{\xrightarrow[n \to +\infty]{}} \mu\,.
$$
Convergence of measures in the Wasserstein space can be characterized using \textit{pseudo-Lipschitz} functions. We say that $f:\mathbb{R}^{d}\rightarrow \mathbb{R}$ is a pseudo-Lipschitz function of degree $r\geq 2$ if there exists a constant $L$ such that for every $\bs{x},\bs{y}\in\mathbb{R}^d$ the following inequality holds: 
$$ \left|f(\bs{x})-f(\bs{y})\right|\leq L \lVert \bs{x}-\bs{y} \rVert \left(1+\lVert \bs{x} \rVert^{r-1}+\lVert\bs{y} \rVert^{r-1}\right). $$
We denote the set of pseudo-Lipschitz functions by $PL_{r}\left(\mathbb{R}^d\right)$. The following classical lemma which can be found in \cite[Lemma 1]{akjouj2023equilibria} summarizes the characterizations of the convergence in the Wasserstein space.
\begin{lemma}
\label{lem:wasserstein}
    Let $\mu_n,\mu\in \mathcal{P}_{r}\left(\mathbb{R}^d\right)$ for $r\geq 2$. The following conditions are equivalent:
    \begin{enumerate}
        \item $\mu_n \stackrel{\mathcal{P}_r\left(\RR^d\right)}{\xrightarrow[n \to +\infty]{}} \mu$,
        \item for all $\varphi\in PL_{r}\left(\mathbb{R}^d\right), \ \int \varphi d\mu_n \rightarrow \int \varphi d\mu$,
        \item $\mu_n \stackrel{w}{\longrightarrow} \mu$ and $\int \lVert \bs{x}\rVert^r \mu_n(d\bs{x}) \rightarrow \int \lVert \bs{x}\rVert^r \mu(d\bs{x}).$
    \end{enumerate}
\end{lemma}

\subsection{Assumptions} We define an AMP algorithm by a triplet $\left(A_n, \mathcal{H}, \left(\bs{u}^0,B_n\right)\right)$, where $A_n$ is a random matrix of size $n\times n$, $\mathcal{H}=\left\{h_k(.,.)\right\}_{k\in\mathbb{N}}$ is a sequence of functions from $\mathbb{R}^{p+1}$ to $\mathbb{R}$, $\bs{u}^0\in \mathbb{R}^{n}$ is an initialization point and $B_n\in\mathbb{R}^{n\times p}$ is a matrix parameter. For our main theorem, the following assumptions are needed.

\begin{enumerate}[label=(A1)]
    \item \label{ass:A1}
     $A_n$ is a normalized elliptic matrix with correlation coefficient $\rho\in [-1,1]$.
\end{enumerate}

\begin{enumerate}[label=(A2)]
    \item \label{ass:A2}
     The random vector $(\bs{u}^0, B_n)\in\mathbb{R}^{p+1}$ is independent of $A_n$ and there exists a vector $(\bar{u},\bar{b}_1,\cdots, \bar{b}_p)$ whose distribution belongs to ${\mathcal P}_r(\mathbb{R}^{p+1})$ such that 
    \begin{equation*}
        \mu^{\bs{u}_n^0, B_n}=\mu^{\bs{u}_n^0,\bs{b}^1_n,\cdots, \bs{b}^p_n} \ \xrightarrow[n\to\infty]{{\mathcal P}_r(\mathbb{R}^{p+1})}  \ \mathcal{L}\left(\left(\bar{u},\bar{b}_1,\cdots, \bar{b}_p\right)\right) \text{ (completely) }.
    \end{equation*}
    We denote $\bs{\bar b}=(\bar{b}_1,\cdots, \bar{b}_p)$. 
\end{enumerate}

\begin{enumerate}[label=(A3)]

\item \label{ass:A3} For all $k\ge 0$, the function $h_k:\mathbb{R}^{p+1}\to \mathbb{R}$ is Lipschitz.
\end{enumerate}
\begin{enumerate}[label=(A4)]
    \item \label{ass:A4}
    For every $k\ge 0$, 
    $
    \mathbb{P}\left(\text{The function } x \mapsto h_k\left(x,\bs{\bar b}\right) \text{is constant} \right) < 1\,.
    $
\end{enumerate}

\begin{enumerate}[label=(A5)]
    \item \label{ass:A5}
    The functions $\partial_{1}h_k$ are continuous $\lambda\otimes \mathbb{P}^{\bar{a}}$- almost everywhere, where $\lambda$ is the Lebesgue measure on $\mathbb{R}$. 
\end{enumerate}





\begin{remark}\phantom{ }

\begin{itemize}
    \item Assumption~\ref{ass:A4} ensures that the covariance matrices $R^k$ defined by the Density Evolution equations in \eqref{eq:DE} are positive definite, and in particular invertible. This is an important assumption used in the proof of Theorem~\ref{th:main} - see \cite[Lemma~2.2]{feng2021unifying}.
    \item Assumption~\ref{ass:A5} is a technicality needed to ensure the convergence of $\left\langle \partial_1 h_k\left(\bs{u}_n^k, B_n\right)\right\rangle_n$ to a deterministic limit.
\end{itemize}
\end{remark}

\subsection{Main result} Recall that $p$ is fixed. We first update the DE equations associated to a matrix parameter $B_n\in \mathbb{R}^{n\times p}$. This simply amounts to replace the scalar $\bar{b}$ in \eqref{eq:DE-init}-\eqref{eq:DE} by the vector $\bs{\bar b} = \left(\bar{b}_1,\cdots, \bar{b}_p\right)$. Let $h_i:\mathbb{R}^{p+1}\to \mathbb{R}$, consider vector $(\bar{u},\bs{\bar b})\in \mathbb{R}^{p+1}$ then Eq. \eqref{eq:DE-init} and \eqref{eq:DE} write
\begin{equation}\label{eq:DE-updated}
R^1 = \mathbb{E}\left[h_0^2\left(\bar{u}, \bs{\bar b}\right)\right]\,,\qquad 
            R^k_{ij} = \begin{cases}
   R^1 & \mbox{ if } i=1, j=1\\
       \mathbb{E}\left[h_{i-1}\left(Z_{i-1}, \bs{\bar b}\right)h_{j-1}\left(Z_{j-1}, \bs{\bar b}\right)\right] & \mbox{ if } i \geq 1 , j\geq 1 \\
   \end{cases}\,
\end{equation}

We can now state the main result of this section.
\begin{theorem}
\label{th:main}
    Let $k\geq 1$, assume \ref{ass:A1}-\ref{ass:A5} and consider a sequence of vectors $\left(\bs{u}^{k}\right)_{k}$ that satisfies the AMP scheme, i.e.
    \begin{equation*}
        \begin{cases}
        \bs{u}^{1} = A_n \ h_0\left(\bs{u}^0, B_n\right), \\
            \bs{u}^{k+1} = A_n \ h_{k}\left(\bs{u}^k, B_n\right) - \rho 
\left\langle \partial_1 h_k\left(\bs{u}^k,B_n\right)\right\rangle_n
h_{k-1}\left(\bs{u}^{k-1}, B_n\right)\,.
        \end{cases}
    \end{equation*}
    Let $(Z_1, \cdots, Z_k)$ be a centered Gaussian vector independent of $(\bar{u},\bs{\bar{b}})$ with covariance matrix $R^k$ given by the updated DE equations~\eqref{eq:DE-updated}. Then the following convergence of the iterates holds:
    \begin{equation}
    \mu^{B, \bs{u}^1,\cdots,\bs{u}^k} \ \stackrel{\mathcal{P}_r\left(\RR^{p+k}\right)}{\xrightarrow[n \to +\infty]{}} \ \mathcal{L}\left(\left(\bs{\bar b}, Z_1, \cdots, Z_k\right)\right) \quad \text{(completely)}\,.
\end{equation}
\end{theorem}
Proof of Theorem \ref{th:main} is provided in Section \ref{sec:proofs}. 

\begin{remark} Proof of Theorem \ref{th:main} crucially relies on the Gaussianity of matrix $A_n$'s entries and an important question would be how to relax this assumption. In Bayati et al. \cite{Bayati_2015}, AMP is extended from a GOE model to a general Wigner matrix (symmetric matrix with i.i.d. entries on and above the diagonal) using an alternative strategy based on combinatorial methods. Adaptation of this combinatorial strategy to an elliptic framework will be the subject of a future work.  
\end{remark}

\begin{remark} Using Lemma~\ref{lem:wasserstein} the convergence result can also be expressed as
\begin{equation*}
    \forall \varphi\in PL_r\left(\mathbb{R}^{p+k}\right), \quad \frac{1}{n}\sum_{i=1}^{n}\varphi\left(b^1_i,\cdots, b^p_i,u^1_i,\cdots, u^k_i\right) \xrightarrow[n\rightarrow\infty]{} \mathbb{E}\left[\varphi\left(\bar{b}_1,\cdots,\bar{b}_p,Z_1,\cdots,Z_k\right)\right].
\end{equation*}
Notice that the sum is over all integers from $1$ to $n$, and thus each iterate vector $\bs{u}^\ell$ ($1\le \ell\le k$) is flattened. One may want to get a more local information,  say the convergence of 
$$
\frac 1{|C^{(n)}|} \sum_{i\in C^{(n)}} \varphi\left(b^1_i,\cdots, b^p_i,u^1_i,\cdots, u^k_i\right)\ ,
$$
where $C^{(n)}$ is a subset of $[n]$.
\end{remark}

Corollary~\ref{cor:partialAMP} generalizes Theorem~\ref{th:main} in this direction. It relies on the following assumption.

\begin{enumerate}[label=(A2$^\prime$)]
    \item \label{ass:A2prime}
    Let $q\geq 1$ be fixed and consider the following partition of $[n]$:
\begin{equation}
\label{eq:partition}
    [n] = C_n^{(1)}\cup \cdots \cup C_n^{(q)} \quad \text{where} 
    \quad \frac{|C_n^{(j)}|}{n} \xrightarrow[n\to\infty]{} c_j \in (0,1) \quad \text{for all}\ j\in [q]\, .
\end{equation} 
There exist $q$ vectors $\left(\bar{u}_j,\bar{b}_{j,1},\cdots,\bar{b}_{j,p}\right)$ with $j\in [q]$ such that:
    \begin{equation*}
        \frac 1{|C_n^{(j)}|}\sum_{i\in  C_n^{(j)}} \delta_{\left(u^0_i,b^1_i,\cdots, b^p_i\right)}\ \xrightarrow[n\rightarrow\infty]{{\mathcal P}_r(\mathbb{R}^{p+1})}\  \mathcal{L}\left(\bar{u}_j,\bar{b}_{j,1},\cdots,\bar{b}_{j,p}\right) \quad \text{(completely)}\,.
    \end{equation*}
\end{enumerate}

\begin{corollary}[blockwise AMP]
    \label{cor:partialAMP}
    Let \ref{ass:A2prime} hold and consider the framework of Theorem~\ref{th:main} except for \ref{ass:A2} (replaced by \ref{ass:A2prime}). Then for all $j\in [q]$
    \begin{equation*}
    \frac{1}{|C_n^{(j)}|}\sum_{i\in C_n^{(j)}} \delta_{\left(b^1_i,\cdots, b^p_i,u^1_i,\cdots, u^k_i\right)}\ \xrightarrow[n\rightarrow\infty]{\mathcal{P}_r\left(\RR^{p+k}\right)} \ {\mathcal L} \left(\bar{b}_{j,1},\cdots,\bar{b}_{j,p},Z_1,\cdots,Z_k\right)\quad \text{(completely)}
    \end{equation*}
where vector $\left(Z_1, \cdots, Z_k\right)$ is defined as in Theorem~\ref{th:main} and is independent from $(\bar{b}_{j,1},\cdots,\bar{b}_{j,p})$.
\end{corollary}
Proof of Corollary~\ref{cor:partialAMP} is postponed to Section~\ref{subsec:partialAMPProof}.


\section{Application to theoretical ecology: equilibria of large LV systems}
\label{sec:ecology}

\subsection{Large Lotka-Volterra systems}

In an ecological system where there are interactions between $n$ species, the dynamics of these species can be modeled by a set of coupled differential equations called a Lotka-Volterra (LV) system. 

Denote by $x_i(t)$ the abundance of species $i$ at time $t$ for $i\in [n]$ and by $\bs{x}(t)=( x_i(t) )_{i\in [n]}$ the vector of abundances of all the species. Denote by $\bs{r}=(r_i)_{i\in [n]}$ the vector of intrinsic growth rates of all the species, and by $\Sigma_n$ the $n\times n$ interaction matrix between the species. 

The LV system is written
$$
\frac{d\bs{x}}{dt} (t) = \bs{x}(t) \odot 
  \left( \bs{r} - \left(I_n - \Sigma_n \right) \bs{x}(t) \right), 
  \quad 
 \bs{x}(0) \in (0,\infty)^n ,  
$$
or equivalently
$$
\frac{d x_i}{dt} (t) = x_i(t) 
  \left( r_i - x_i +\sum_{j\in [n]} \Sigma_{ij} x_j(t) \right)\,, 
  \quad 
 x_i(0)>0\qquad \textrm{for all}\ i\in [n]\, .
$$
Here $\Sigma_{ij}$ represents the effect of species $j$ on the growth of species $i$. Notice that if $\Sigma_n=0$ (no interactions), then each species is described by a logistic differential equation. The general properties of a LV system are well-known, see for instance \cite[Chapter 3]{takeuchi1996global}, notice in particular that $\bs{x}_n(t)> 0$ (componentwise) for all $t\ge 0$ if $\bs{x}(0) \in (0,\infty)^n$.

We are interested in the regime where $n$ is large and $\Sigma_n$ is random. Often, the real values of $\Sigma_n$ are out of reach and an alternative is to choose a random model which statistical properties would reflect a partial knowledge on the ecological interaction network. Among the various matrix models, the $\Elliptic(n,\rho)$ model represents a good trade-off between complexity and tractability, see 
\cite{allesina2012stability,allesina2015stability,akjouj2022complex}. We will 
therefore assume that 
\begin{equation}\label{ass:Sigma}
\kappa \sqrt{n} \Sigma_n \sim \Elliptic(n,\rho)\, ,
\end{equation}
Otherwise stated, $\Sigma_n=\frac{A_n}{\kappa}$ where $A_n$ is a normalized elliptic matrix and $\kappa>0$ is an extra parameter.
In this case, $\|\Sigma_n\| ={\mathcal O}(1)$ as $n\to \infty$ and the interaction matrix $\Sigma_n$ will have a macroscopic effect on the LV system as $n\to\infty$.

\subsection{Existence of a stable and unique equilibrium} In \cite{clenet2022equilibrium}, sufficient conditions are provided so that system \eqref{eq:LV-system} eventually admits a unique and stable equilibrium. 
\begin{proposition}[Prop. 2.3 in \cite{clenet2022equilibrium}]\label{prop:existence-equilibrium}
    Consider system \eqref{eq:LV-system} where $\kappa \sqrt{n} \Sigma_n \sim \Elliptic(n,\rho)$ and suppose that $\kappa>\sqrt{2(1+\rho)}$. Then almost surely (a.s.) eventually there exists a unique and stable equilibrium $\bs{x}^\textrm{\sc eq}=(x_i^\textrm{\sc eq}(n))_{i\in [n]}$. Otherwise stated, with probability one there exists $N$ such that for all $n\ge N$, there exists a unique $\bs{x}^\textrm{\sc eq}\in \mathbb{R}^n$ 
    such that
    $$
    \bs{x}(t) \xrightarrow[t\to\infty]{}\bs{x}^\textrm{\sc eq}\, ,
    $$
    where $\bs{x}(t)$ solves \eqref{eq:LV-system}.
\end{proposition}

\begin{remark}
    The fact that $\bs{x}(t)>0$ for all $t>0$ only implies that $\bs{x}^\textrm{\sc eq}\ge 0$. A vanishing component of $\bs{x}^\textrm{\sc eq}$ represents a vanishing species (whose abundance is zero at equilibrium).   
\end{remark}

\begin{remark}[extension of the definition of $\bs{x}^\textrm{\sc eq}$]
Notice that in Proposition \ref{prop:existence-equilibrium}, the equilibrium $\bs{x}^\textrm{\sc eq}$ is eventually defined. In fact, standard arguments yield 
$$
 \frac {\|A_n\|}{\kappa} \quad \xrightarrow[n\to\infty]{a.s.} \quad \frac{\sqrt{2(1+\rho)}}{\kappa} <1
$$
by assumption over $\kappa$. If the condition $ \frac {\|A_n\|}{\kappa} <1$ is met, which happens a.s. eventually, then the existence of $\bs{x}^\textrm{\sc eq}$ is granted by Takeuchi's result \cite[Th. 3.2]{takeuchi1996global}. We extend the definition of $\bs{x}^\textrm{\sc eq}$ by setting
\begin{equation}\label{eq:extension-xstar}
    \bs{ x}^\star =\begin{cases}
        \ \bs{x}^\textrm{\sc eq} &\textrm{if}\quad \frac{\| A_n\|}{\kappa}<1\,,\\
        \ 0&\textrm{else}\,.
    \end{cases}
\end{equation}
With a slight abuse of notation, one may denote $\bs{x}^\star=\bs{x}^\textrm{\sc eq}\, \mathbf{1}_{\{ \|A_n\|/\kappa\ <\ 1\}}$.    
\end{remark}

\subsection {Statistical properties of the LV equilibrium $\bs{x}^\star$}

Once the existence of the equilibrium is granted, we shall explore its statistical properties and address questions such as: What is the proportion of surviving species at equilibrium? What is the distribution of surviving species? etc.
In this regard, a key device will be the study of the empirical probability measure 
$$\mu^{\bs{x}^\star}=\frac 1n \sum_{i\in [n]} \delta_{x_i^\star(n)}\ ,
$$
and the design of an appropriate AMP algorithm. 

We first introduce a system of three equations whose solutions will play a key role in describing the statistical properties of the equilibrium.
\begin{lemma}
\label{lemma:main-system} Let $\rho\in [-1,1]$ and suppose that $\kappa>(1+\rho)/\sqrt{2}$. Consider two independent real random variables $\bar{Z}$ and $\bar{r}$ where $Z\sim {\mathcal N}(0,1)$ and $\bar r\ge 0$ with finite second moment and $\mathcal{L}(\bar{r})\neq \delta_0$. Then the system of 
equations 
\begin{subequations}
\label{eq:sys} 
\begin{align}
\kappa &= \delta + \rho \frac{\gamma} \delta\,, \label{eq:sys-delta} \\ 
\sigma^2 &= \frac{1}{\delta^2} 
   \mathbb{E}\left[\left( \sigma \bar{Z} + \bar{r} \right)_+^2\right]\,,  \label{eq:sys-sigma}\\ 
\gamma &= \mathbb{P} \left[ \sigma \bar{Z} + \bar{r} > 0 \right]\,,\label{eq:sys-gamma}
\end{align}
\end{subequations}
admits an unique solution $(\delta,\sigma,\gamma)$ in
$(1/\sqrt{2},\infty) \times (0,\infty) \times (0,1)$.  
\end{lemma}

Proof of Lemma \ref{lemma:main-system} is deferred to Appendix \ref{app:proof-system}. We follow the lines of the corresponding proof for the symmetric matrix case \cite[Section 3.2]{akjouj2022complex} but the case $\rho<0$ requires new arguments.

\begin{remark}
    Notice that condition $\kappa>(1+\rho)/\sqrt{2}$ in Lemma \ref{lemma:main-system} is weaker than condition
    $\kappa > \sqrt{2(1+\rho)}$ provided in Proposition \ref{prop:existence-equilibrium} unless $\rho=-1$. Otherwise stated if $\kappa$ satisfies the condition
    $$
    \frac {1+\rho}{\sqrt{2}} <\kappa \le \sqrt{2(1+\rho)}\,,\quad (\rho>-1)
    $$
    the system may admit a unique solution but the existence of a stable equilibrium is not granted.
\end{remark}
We can now state the main result of this section.

\begin{theorem}
\label{th:main-LV} 
Let $A_n$ be a normalized elliptic matrix, $\bs{r}\in \mathbb{R}^n$ a random vector independent from $A_n$ satisfying:
$$
\mu^{\bs{r}}  \ \xrightarrow[n\to\infty]{\mathcal{P}_2(\RR)} \ {\mathcal L}(\bar{r})\quad \text{(completely)}\,,
$$
where $\bar r \geq 0$ is a real valued random variable with finite second moment and $\mathcal{L}(\bar{r})\neq \delta_0$. 
Let $\bar{Z}$ be a  $\mathcal{N}(0, 1)$ random variable independent from $\bar r$.


Let $\kappa>\sqrt{2(1+\rho)}$ and consider the LV system \eqref{eq:LV-system} where $\Sigma_n=\frac {A_n}\kappa$. 
Let $\bs{x}^\star$ be defined by \eqref{eq:extension-xstar} and 
$(\delta, \sigma, \gamma)\in (1/\sqrt{2},\infty) \times (0,\infty) \times (0,1)$ be the solution of \eqref{eq:sys} in Lemma \ref{lemma:main-system}, then
\begin{equation}
\label{cvg-muN} 
\mu^{\bs{x}^\star} \ \xrightarrow[n\to\infty]{\mathcal{P}_2(\RR)} \
 \pi:=\mathcal{L}\left( \left( 1 + \rho \gamma/\delta^2 \right) 
  \left( \sigma \bar Z + \bar r \right)_+ \right) \quad \text{(completely)}\ .
\end{equation} 
\end{theorem}

Proof of Theorem \ref{th:main-LV} is outlined in Section \ref{subsec:AMP-for-LV}. It closely follows the strategy developed in \cite{akjouj2023equilibria} in the context of a symmetric interaction matrix (see in particular the outline of the proof in \cite[Section 3.1]{akjouj2023equilibria}). This strategy is adapted to the elliptic case with the help of Theorem \ref{th:AMP_conv} and the existence of a unique solution to \eqref{eq:sys}.

\begin{remark}[Proportion of surviving species]
\label{rem:surviving-species}
    Strictly speaking, the proportion of surviving species at equilibrium is given by:
    $$
    \mu^{\bs{x}^\star}(0,\infty) = \frac 1n \sum_{i\in [n]} \bs{1}_{(0,\infty)}(x_i^\star)\, .
    $$
    As a consequence, convergence \eqref{cvg-muN} in Theorem \ref{th:main-LV} does not apply for $x\mapsto \bs{1}_{(0,\infty)}(x)$ is not continuous at zero, a discontinuity point of the limiting cumulative function. However for any continuous function $f_{\varepsilon}$ satisfying 
    $$
    f_{\varepsilon}(x) = 
    \begin{cases}
         0 & \textrm{for}\ x\le 0\\
        1& \textrm{for}\ x\ge \varepsilon\
    \end{cases}$$
 for a small $\varepsilon > 0$, one has 
    $\frac 1n \sum_i f_{\varepsilon}(x_i^\star)\xrightarrow[n\to\infty]{} \mathbb{E} f_{\varepsilon} \left[\left( 1 + \rho \gamma/\delta^2 \right) 
  \left( \sigma \bar Z + \bar r \right)_+\right]$
  and
  $$
  \mathbb{E} f_{\varepsilon} \left[\left( 1 + \rho \gamma/\delta^2 \right) 
  \left( \sigma \bar Z + \bar r \right)_+\right]\xrightarrow[\varepsilon\to 0]{} \gamma=\mathbb{P}( \sigma \bar Z + \bar r>0) \, .$$
  Hence $\gamma$ appears as a good approximation of the proportion of surviving species. This is confirmed by simulations, see Figure~\ref{fig:prop}.
\end{remark}

\begin{remark}[Distribution of surviving species] 
\label{rem:distribution-surv-species} Denote by $\bs{s}(\bs{x}^\star)$ the subvector of $\bs{x}^\star$ with the positive components of $\bs{x}^\star$. Its dimension $|\bs{s}(\bs{x}^\star)|$ is random and the distribution of the surviving species is given by:
$$
\mu^{\bs{s}(\bs{x}^\star)} = \frac 1{|\bs{s}(\bs{x}^\star)|}\sum_{i\in [|\bs{s}(\bs{x}^\star)|]} \delta_{[\bs{s}(\bs{x}^\star)]_i}\ .
$$
A formal convergence of $\mu^{\bs{s}(\bs{x}_n^\star)}$ is out of reach (see the arguments in Remark \ref{rem:surviving-species}) but a good proxy should be 
\begin{equation}\label{eq:law-surviving}
{\mathcal L} \left( \left( 1 + \rho \gamma/\delta^2 \right) 
  \left( \sigma \bar Z + \bar r \right)_+ \ \bigg| \  \sigma \bar{Z}+\bar{r}>0\right)\ ,
\end{equation}
the density of which is explicit. Let
$$
f_{\sigma \bar{Z} +\bar{r}}(y) = \int_{\mathbb{R}} \frac{e^{-\frac{(y-r)^2}{2\sigma^2}}}{\sqrt{2\pi}\, \sigma} \mathbb{P}_{\bar{r}}(dr)\quad \textrm{and}\quad 1+\rho\frac{\gamma}{\delta^2} = \frac{\kappa}{\delta} \ ,
$$
then the density of \eqref{eq:law-surviving} denoted by $f_{\textrm{surv}}$ is written
\begin{equation}\label{eq:density-surviving}
f_{\textrm{surv}}(y)\ =\ \frac{\delta}{\kappa} f_{\sigma \bar{Z} +\bar{r}}\left(\frac{\delta\,y}{\kappa}\right) \frac{\bs{1}_{(y>0)}}{\gamma}\, .
\end{equation}
One can now easily express the relation between $\pi$ as defined in \eqref{cvg-muN} and $f_{\textrm{surv}}$:
\begin{equation}\label{eq:pi-fsurv}
\pi(dy) = \gamma f_{\textrm{surv}}(y)\, dy + (1-\gamma)\delta_0(dy)\, .
\end{equation}
Notiece that if the r.v. $\bar{r}$ is constant then $f_{\textrm{surv}}$ is the density of a truncated Gaussian distribution. 
\end{remark}
Simulations show a very good fit between this distribution and the histogram associated to $\mu^{\bs{s}(\bs{x}^\star)}$ for large $n$, see Fig.~\ref{fig:dist}.

\begin{figure}
\centering
\begin{subfigure}{0.45\textwidth}
    \includegraphics[scale=0.45]{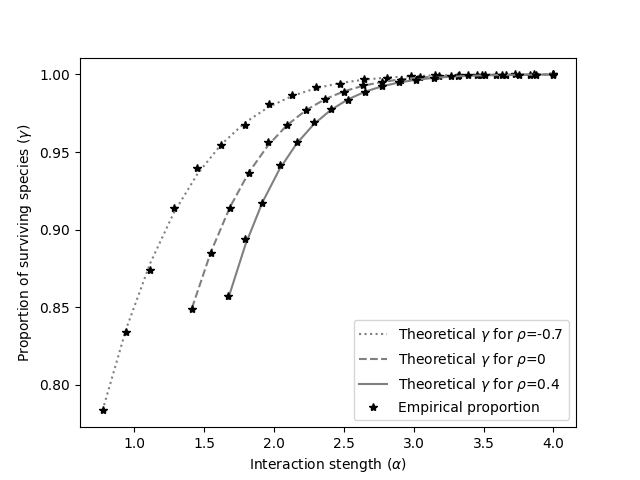}
    \caption{Experimental proportions of surviving species vs theoretical values $\gamma$ for three correlation coefficients $\rho= -0.7, 0, 0.4$ w.r.t. the interaction strength ($\kappa$).}
    \label{fig:prop}
\end{subfigure}
\hspace{0.45cm}
\begin{subfigure}{0.45\textwidth}
    \includegraphics[scale=0.45]{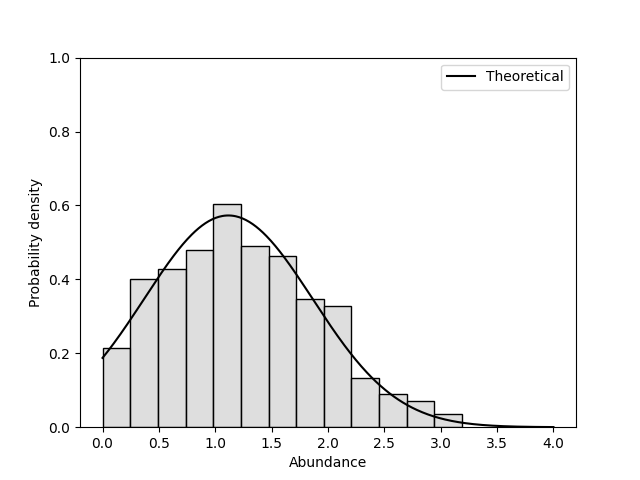}
    \caption{Histogram of positive abundances vs the density function $f_{\textrm{surv}}$ described in \eqref{eq:density-surviving} for $\rho=0.4$ with the interaction strength fixed to $\kappa=2$.}
    \label{fig:dist}
\end{subfigure}

\begin{subfigure}{0.45\textwidth}
    \includegraphics[scale=0.45]{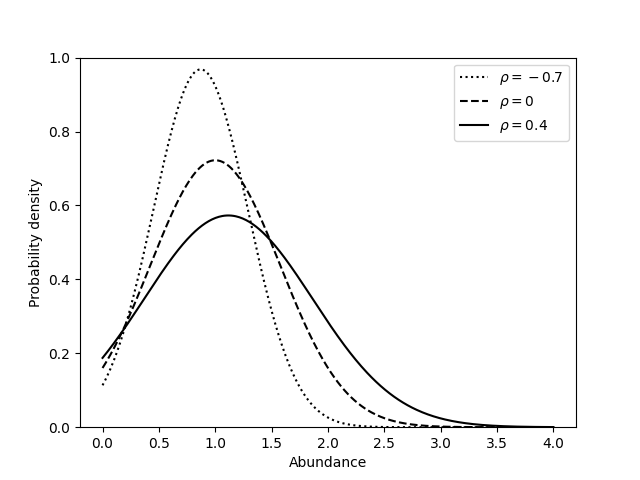}
    \caption{Plot of the density function $f_{\textrm{surv}}$ for $\rho = -0.7, 0, 0.4$.}
    \label{fig:truncdist}
\end{subfigure}
        
\caption{Comparison between the theoretical solution of the fixed point equations \eqref{eq:sys} and their empirical Monte Carlo counterpart obtained by computing equilibria $\bs{x}^\star$ for various realizations of matrix $A$. Every $\bs{x}^\star$ is the solution of a Linear Complementarity Problem (see \eqref{eq:LCP-first-sight}) and is thus computed by Lemke algorithm. For Figure~\ref{fig:prop} and Figure~\ref{fig:dist} we chose a matrix of size $200$ and we fixed the number of Monte Carlo experiments to $100$ and $500$ respectively. }
\label{fig:figures}
\end{figure}


\subsection{Propagation of chaos}
Combining Theorem \ref{th:main-LV} with propagation of chaos type results
\cite{sznitman89}, we are able, with extra exchangeability assumptions
on the vector $\bs{r}\in \mathbb{R}^n$, to describe the limiting behaviour of
individual abundances. 

We obtain two kinds of results. If $\bs{r}$ is exchangeable, then the distribution of every individual abundance converges toward the same limit given by $\pi$ in \eqref{cvg-muN}, see Corollary \ref{coro:chaos-prop-I}. If $\bs{r}$ is blockwise exchangeable (to be defined), then within each block each abundance can have a specific limit, see Theorem~\ref{th:chaos-prop-II}, which may differ from $\pi$.

\begin{corollary}
\label{coro:chaos-prop-I} 
Consider the framework of Theorem \ref{th:main-LV} and assume moreover that  
vector $\bs{r}$ is exchangeable. Let $\bs{x}^\star =(x_i^\star(n))$ be defined by \eqref{eq:extension-xstar} and recall the definition of the distribution $\pi$:
$$
 \pi=\mathcal{L}\left( \left( 1 + \rho \gamma/\delta^2 \right) 
  \left( \sigma \bar Z + \bar r \right)_+ \right)\, .
$$
Then $( x_1^\star(n), \cdots,  x_n^\star(n))$ is an exchangeable sequence and for any fixed $K\ge 1$,  
$$
( x_1^\star(n),\cdots,  x_K^\star(n)) \xrightarrow[n\to\infty]{\mathcal L} \pi^{\otimes K}\, .
$$
\end{corollary}
Proof of Corollary \ref{coro:chaos-prop-I} is postponed to Section \ref{subsec:proof-chaos-prop-I}.
\begin{remark}
    This result should be compared to Geman and Hwang \cite[Theorem 3]{geman1982chaos}. 
\end{remark}

For the next result, we need some extra definitions. Let $q\ge 1$ be a fixed integer. Consider $q$ sequences $n_1(n),\cdots, n_q(n)$ satisfying
\begin{equation}\label{eq:nj}
n_1+\cdots+ n_q = n \qquad \text{and}\qquad \frac{n_j}n \xrightarrow[n\to\infty]{} c_j\in (0,1)\,,\quad j\in [q]\,.
\end{equation}
Consider the following partition of $[n]$:
\begin{equation}\label{eq:Cj}
    \begin{cases}
C_n^{(1)}&=\{ 1,\cdots, n_1\}\,,\\
C_n^{(j)}&=\{n_1+\cdots +n_{j-1}+1,\ \cdots\ , n_1+\cdots+n_j\}\,,\quad 1<j\le q\, 
\end{cases}
\end{equation}
so that $$
[n]=\bigcup_{i=1}^q C_n^{(j)}\qquad \textrm{and}\qquad |C_n^{(j)}|=n_j\,,\quad j\in [q]\,.
$$

Any $\mathbb{R}^n$-valued vector $\bs{v}$ can be decomposed into $q$ $\mathbb{R}^{n_j}$-valued subvectors $\bs{v}^{(j)}$:
$$
\bs{v}=(\bs{v}^{(j)},\ j\in [q])\qquad \text{where}\qquad \bs{v}^{(j)} = (v_k)_{k\in C_n^{(j)}}\, .
$$
Given a permutation $\sigma_j\in {\mathcal S}_{n_j}$, we denote by $\bs{v}^{(j,\sigma_j)}$ the vector $\bs{v}^{(j)}$ where each component has been permuted according to $\sigma_j$.

Consider now $\sigma=(\sigma_1,\cdots, \sigma_q)$ where $\sigma_j\in {\mathcal S}_{n_j}$. Given a vector $\bs{v}\in \mathbb{R}^n$ we denote by $\bs{v}^\sigma$ the vector
$$
\bs{v}^\sigma = \left( \bs{v}^{(j,\sigma_j)};\ j\in [q]\,\right)\,.
$$ 
\begin{example} Consider $n=6$ and a number of blocks $q=2$ such that $n_1 = 4$ and $n_2 = 2$. Let $\bs{v} = \left(v_1,v_2,v_3,v_4,v_5,v_6\right)^\top$ and consider two permutations $$
\sigma_1 = \begin{pmatrix}
    1 & 2 & 3 & 4 \\
    3 & 4 & 1 & 2
\end{pmatrix}\in \mathcal{S}_4\quad \textrm{and}\quad\sigma_2 = \begin{pmatrix}
    1 & 2  \\
    2 & 1
\end{pmatrix}\in \mathcal{S}_2\, .
$$
Then $\bs{v}^{(1,\sigma_1)} = \left(v_3,v_4,v_1,v_2\right)^\top$ and $\bs{v}^{(2,\sigma_2)} = \left(v_6,v_5\right)^\top$. Let $\sigma = \left(\sigma_1, \sigma_2\right)$, then
$ \bs{v}^\sigma = \left(v_3,v_4,v_1,v_2,v_6,v_5\right)^\top. $
\end{example}
\begin{definition}\label{def:blockEx} A random $\mathbb{R}^n$-valued vector $\bs{v}$ is blockwise exchangeable with respect to the partition defined in \eqref{eq:Cj} if for any $\sigma=(\sigma_1,\cdots, \sigma_q)\in {\mathcal S}_{n_1}\times \cdots\times {\mathcal S}_{n_q}$,
    $$
    \bs{v}^{\sigma}\ \eqlaw\  \bs{v}\, .
    $$
    Otherwise stated, for any bounded continuous test functions $\varphi_j:\mathbb{R}^{n_j}\to \mathbb{R}$ with $j\in [q]$, 
    $$
    \mathbb{E}\left(\varphi_1\left(\bs{v}^{(1)}\right)\times \cdots \times \varphi_q\left(\bs{v}^{(q)}\right)\right)
    = \mathbb{E}\left( \varphi_1\left(\bs{v}^{(1,\sigma_1)}\right)\times \cdots \times \varphi_q\left(\bs{v}^{(q,\sigma_q)}\right) \right)\,.
    $$
    If there is no confusion, we simply say that $\bs{v}$ is blockwise exchangeable.
\end{definition}

We are now in position to state our final result. 
\begin{theorem}
    \label{th:chaos-prop-II}
    Consider the framework of Theorem \ref{th:main-LV} and let $q\ge 1$ be a fixed integer. Let $(n_j, j\in [q])$ and $(C_n^{(j)},j\in [q])$ be given by \eqref{eq:nj}-\eqref{eq:Cj}. Let $\bs{x}^\star$ be defined in \eqref{eq:extension-xstar}. Assume that $\bs{r}$ is blockwise exchangeable and that for all $j\in [q]$
    \begin{equation*}
    \label{eq:r-block-conv}
    \mu^{\bs{r^{(j)}}} \xrightarrow[n\to\infty]{{\mathcal P}_2(\mathbb{R})}\bar{r}_j\quad \text{(completely)}     
    \end{equation*}
    where $\bar{r}_j\ge  0$ is a random variable with finite second moment.

    Then for any sequence $\psi_n\in C_n^{(j)}$ where $j\in [q]$ is fixed,
    $$
    x^\star_{\psi_n} \xrightarrow[n\to\infty]{\mathcal L}  \pi_j:=\mathcal{L}\left( \left( 1 + \rho \gamma/\delta^2 \right) 
  \left( \sigma \bar Z + \bar r_j \right)_+ \right)\, .
    $$
    Moreover, let $k_1,\cdots, k_q\ge 1$ be fixed integers and consider subsets 
    $$
    {\mathcal K}_n^{(j)}\subset C_n^{(j)} \quad \text{with}\quad  
    |{\mathcal K}_n^{(j)}|=k_j\quad \text{and}\quad k=k_1+\cdots +k_q\ ,
    $$
    then the $\mathbb{R}^k$-valued vector
    $$
    \bs{ x}^\star_{[k_1,\cdots, k_q]} := \left( x^\star_\ell\,,\ \ell\in {\mathcal K}_n^{(1)}\cup\cdots \cup {\mathcal K}_n^{(q)}\right)
    $$
    satisfies
    $$
    \bs{\tilde x}^\star_{[k_1,\cdots, k_q]}
    \xrightarrow[n\to \infty]{\mathcal L} \prod_{j=1}^q \pi_j^{\otimes k_j}\, . 
    $$
\end{theorem}
Proof of Theorem \ref{th:chaos-prop-II} is postponed to Section \ref{subsec:proof-chaos-prop-II}.
\begin{remark}[Global versus local distributions]
    Unlike the case where the intrinsic growth rates vector $\bs{r}$ is exchangeable (see Corollary~\ref{coro:chaos-prop-I}), notice now that each block of the equilibrium vector $\bs{x}^{\star}$ converges to a different law which locally depends on the structure of $\bs{r}$. In particular the $j$-th block of $\bs{x}^{\star}$ converges to
    $$  \pi_j=\mathcal{L}\left( \left( 1 + \rho \gamma/\delta^2 \right) \left( \sigma \bar Z + \bar r_j \right)_+ \right), $$
    which is different than the overall asymptotic behaviour of $\bs{x}^\star$,
    $$  \pi=\mathcal{L}\left( \left( 1 + \rho \gamma/\delta^2 \right) \left( \sigma \bar Z + \bar r \right)_+ \right). $$
    Simulations based on a three-block piece-wise constant vector $\bs{r}=(\bs{r}^{(1)},\bs{r}^{(2)},\bs{r}^{(3)})$ are provided in Figure~\ref{fig:exchangeability}.
\end{remark}

\begin{remark}[Distribution $\pi$ is a mixture of the $\pi_j$'s]
    Recall the definitions of the $q$ random variables $(\bar{r}_j)_{j\in [q]}$ in Theorem~\ref{th:chaos-prop-II} and the definition of $\bar{r}$ in Theorem~\ref{th:main-LV}, we can notice that the law of $\bar{r}$ is the mixture the laws of $(\bar{r}_j)_{j\in [q]}$ with coefficients $(c_j)_{j\in [q]}$, i.e.
    \begin{equation*}
        \mathcal{L}(\bar{r}) = \sum_{j=1}^{q} c_j\mathcal{L}(\bar{r}_j).
    \end{equation*}
    This also means that the limiting distribution of the whole equilibrium vector $\bs{x}^{\star}$ is a mixture of laws, i.e.
    \begin{equation*}
        \mu^{\bs{x}^\star} \ \xrightarrow[n\to\infty]{\mathcal{P}_2(\RR)}\sum_{j=1}^{q} c_j \pi_j
        \quad \text{where}\quad \pi_j=\mathcal{L}\left( \left( 1 + \rho \gamma/\delta^2 \right) \left( \sigma \bar Z + \bar r_j \right)_+ \right).
    \end{equation*}
\end{remark}

\begin{remark}[surviving species within block $j$]
    From an empirical point of view, simulations easily provide the number of surviving species within a block $j$ of size $n_j$, that is
    $$
    \frac{\#\{ x^*_\ell>0, \ell\in C^{(j)}_n\}}{n_j}\ ,
    $$
    and the value of their positive abundance. Following Remarks \ref{rem:surviving-species} and \ref{rem:distribution-surv-species}, Theorem \ref{th:chaos-prop-II} provides their analytical counterparts. The quantity
    $$
    \gamma_j =\mathbb{P} (\sigma\bar{Z} +\bar{r}_j>0) 
    $$
    is a good approximation for the proportion of surviving species within block $j$ and the density
$$
f^{j}_{\textrm{surv}}(y)\ =\ \frac{\delta}{\kappa} f_{\sigma \bar{Z} +\bar{r}_j}\left(\frac{\delta\,y}{\kappa}\right) \frac{\bs{1}_{(y>0)}}{\gamma_j}\quad \text{where}\quad 
f_{\sigma \bar{Z} +\bar{r}_j}(y) = \int_{\mathbb{R}} \frac{e^{-\frac{(y-r)^2}{2\sigma^2}}}{\sqrt{2\pi}\, \sigma} \mathbb{P}_{\bar{r}_j}(dr),
$$
for the distribution of the surviving species in block $j$. 

One can notice that $f_{\textrm{surv}}$ is a mixture of the $f^j_{\textrm{surv}}$'s:
$$
f_{\textrm{surv}}(y) = \sum_{j=1}^q  \frac{c_j \gamma_j}{\gamma} f^j_{\textrm{surv}}(y)\quad \text{with}\quad 
\sum_{j=1}^q \frac{c_j \gamma_j}{\gamma} =1\, .
$$
Based on a three-block piece-wise constant vector $\bs{r}=(\bs{r}^{(1)},\bs{r}^{(2)},\bs{r}^{(3)})$, the densities $f^1_{\textrm{surv}}$, $f^2_{\textrm{surv}}$ and $f^3_{\textrm{surv}}$ are compared to the corresponding simulation based histograms in Figure~\ref{fig:exchangeability}.
\end{remark}
\begin{figure}
\centering
\begin{subfigure}{0.45\textwidth}
    \includegraphics[scale=0.45]{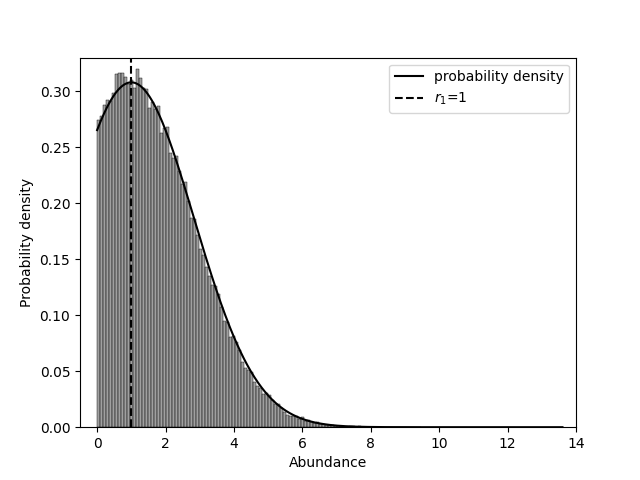}
    \caption{The density $f_{\text{surv}}^1$ compared to the histogram of surviving species in block 1.}
    \label{fig:r1}
\end{subfigure}
\hspace{0.45cm}
\begin{subfigure}{0.45\textwidth}
    \includegraphics[scale=0.45]{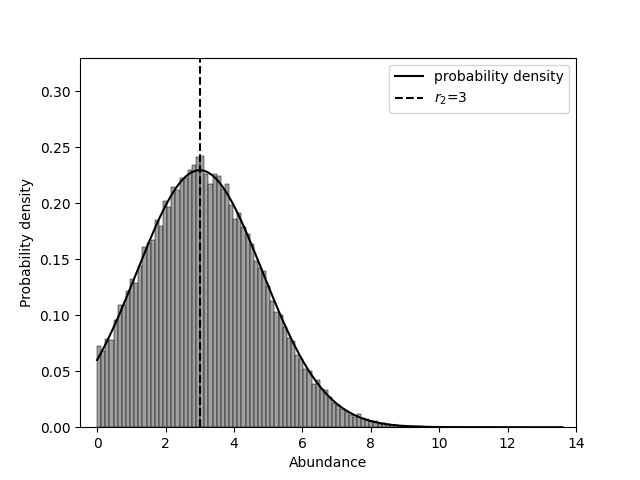}
    \caption{The density $f_{\text{surv}}^2$ compared to the histogram of surviving species in block 2.}
    \label{fig:r2}
\end{subfigure}

\begin{subfigure}{0.45\textwidth}
    \includegraphics[scale=0.45]{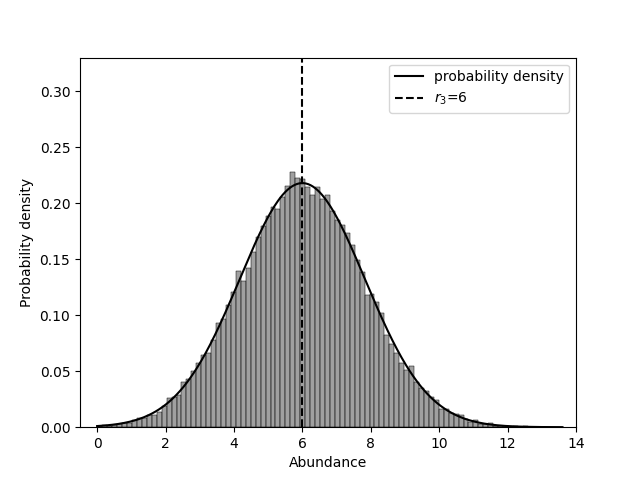}
    \caption{The density $f_{\text{surv}}^3$ compared to the histogram of surviving species in block 3.}
    \label{fig:r3}
\end{subfigure}
\hspace{0.45cm}
\begin{subfigure}{0.45\textwidth}
    \includegraphics[scale=0.45]{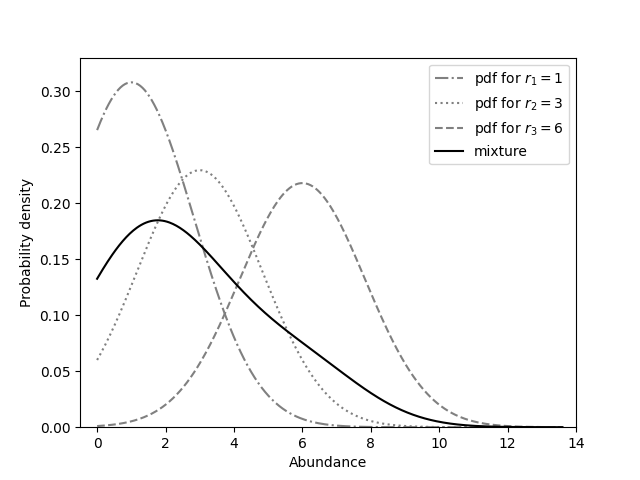}
    \caption{The Distribution of $\bs{x}^\star$ as a mixture of the three blocks' distributions.}
    \label{fig:r123}
\end{subfigure}
        
\caption{The vector $\bs{r}$ is decomposed into three consecutive blocks, for each block we fix a constant value $r_1,r_2,r_3 = 1,3,6$ respectively, we also choose different sizes of the blocks to be $n/2, 3n/10, n/5$. We then solve the Linear Complementarity Problem problem (see \eqref{eq:LCP-first-sight}) with the help of Lemke algorithm for $\rho=0 $ and $ \kappa=2$.}
\label{fig:exchangeability}
\end{figure}

\section{Proofs of Theorem \ref{th:main} and Corollary \ref{cor:partialAMP}}
\label{sec:proofs}
In Section \ref{subsec:elliptical-results} we provide various results related to elliptic random matrices. 
Sections \ref{subsec:some-preparation}--\ref{subsec:end-of-proof-main} are devoted to the proof of Theorem \ref{th:main}. After introducing new notations in Section \ref{subsec:some-preparation}, we provide an adaptation of Bolthausen conditioning argument to elliptic random matrices in Section \ref{subsec:bolthausen}, see Propositions \ref{prop:bolt} and \ref{prop:bolt++}. This represents the crux of the proof of Theorem \ref{th:main} and our main contribution to this section. Section \ref{subsec:end-of-proof-main} is devoted to the end of proof of Theorem  
\ref{th:main} and closely follows \cite{feng2021unifying}. Proof of Corollary \ref{cor:partialAMP} is established in Section \ref{subsec:partialAMPProof}.

\subsection{Preliminary results on elliptic matrices}
\label{subsec:elliptical-results}
Let $X$ be a $n\times n$ matrix with independent ${\mathcal N}(0,1)$ entries. 
If the $n\times n$ matrices $G$ and $\widetilde G$ satisfy
$$
G \ \eqlaw\ \frac{X + X^\top}{\sqrt{2}}\qquad \text{and}\qquad \widetilde G \ \eqlaw\ \frac{X - X^\top}{\sqrt{2}}\, ,
$$
then we say that $G$ is a GOE matrix and $\widetilde G$ an antisymmetric GOE
matrix. 

From the definition \ref{def:elliptic-model} of an elliptic matrix, it is 
easy to check that a matrix $M \sim \Elliptic(n,\rho)$ for $\rho \in [-1,1]$ 
can be characterized as 
\begin{equation}
\label{ellip-goe} 
M \ \eqlaw\ \sqrt{\frac{1+\rho}{2}} G + \sqrt{\frac{1-\rho}{2}} \widetilde G, 
\end{equation} 
where $G$ is a GOE matrix, $\widetilde G$ is an antisymmetric GOE
matrix, and $G \indep \widetilde G$.

We begin by two elementary results on GOE matrices:
\begin{lemma}
\label{lemma:elliptic-misc} 
Let $G$ and $\widetilde G$ be respectively a symmetric and an antisymmetric 
$n\times n$ GOE matrix. Consider two deterministic vectors $\bs{u}, \bs{v} \in \RR^n$, then:
\begin{eqnarray*}
&(i)& \EE\, G \bs{u}\bs{v}^\top G = (\bs{v}^\top \bs{u}) I_n + \bs{v}\bs{u}^\top\quad \textrm{and}\quad \EE\, \widetilde G \bs{u}\bs{v}^\top \widetilde G = - (\bs{v}^\top \bs{u}) I_n + \bs{v}\bs{u}^\top\, ,\\
&(ii)& G \bs{u} \sim \cN\left( 0, I_n + \bs{u} \bs{u}^\top\right)\quad \textrm{and}\quad \widetilde G\bs{u} \sim \cN\left( 0, I_n - \bs{u}\bs{u}^\top\right)\,.
\end{eqnarray*}
\end{lemma} 

\begin{proof}
We prove the two statements for $G$, the corresponding results for $\widetilde G$ can be shown similarly. 
For $(i)$, we start by writing 
$\EE [G \bs{u}\bs{v}^\top G]_{kk} = \sum_{ij} \EE G_{ki} G_{kj} u_i v_j = 
\sum_i \EE G_{ki}^2 u_i v_i = u_k v_k + \bs{v}^\top \bs{u}
$, 
and $\EE [G \bs{u}\bs{v}^\top G]_{kl} = \sum_{ij} \EE G_{ki} G_{jl} u_i v_j = u_l v_k$  
for $k\neq l$. Thus, $\EE G \bs{u}\bs{v}^\top G = (v^\top u) I_n + \bs{v}\bs{u}^\top$. 
We now prove $(ii)$. As in the proof of \cite[Lemma 6.14]{feng2021unifying}, let us complete the
vector $\bs{u}$ in a deterministic orthogonal matrix 
$U = \begin{bmatrix} \bs{u} \ \widetilde U \end{bmatrix}$. By the orthogonal 
invariance of GOE matrices, we have  
\[
G \bs{u} \eqlaw U G U^\top \bs{u} = U G e_1 \sim 
\cN\left( 0, U (e_1 e_1^\top + I_n) U^\top \right) 
  = \cN\left(0, I_n + \bs{u}\bs{u}^\top\right)\,,
\]
hence the desired result.
\end{proof}


Using these results, we now have the two following propositions of elliptic matrices. 
\begin{proposition}
\label{prop:mainIndep}
Let $M\sim \Elliptic(n,\rho)$. Let $\bs{q}\in \RR^n$ and $U\in \RR^{n\times k}$ 
be deterministic with $\bs{q}^\top U = 0$. Let $P\in \RR^{n\times n}$ be a
deterministic matrix that satisfies $PU = 0$. Then
\[
(M-\rho PM^\top)\bs{q} \quad \indep\quad  
\begin{bmatrix}
(M - \rho M^\top)U \\
U^\top M^\top U
\end{bmatrix}.
\]
\end{proposition}
\begin{proof}
Since the considered quantities form a Gaussian vector, it is enough to show that
$(M-\rho PM^\top)\bs{q}$ is decorrelated from all columns of the two
matrices $(M - \rho M^\top)U$ and $U^\top M^\top U$.  
To this end, we use the characterization~\eqref{ellip-goe} and write 
\[
\begin{bmatrix} 
  M \\
M-\rho M^\top 
\end{bmatrix} 
\quad  \eqlaw\quad  
\begin{bmatrix} 
  \sqrt{\frac{1+\rho}{2}} G + \sqrt{\frac{1-\rho}{2}} \widetilde G, \\
(1-\rho) \sqrt{\frac{1+\rho}{2}} G + (1+\rho) \sqrt{\frac{1-\rho}{2}} 
  \widetilde G 
\end{bmatrix} . 
\]
Let $\bs{u}$ be any column of  $U$. We first show that 
$(M-\rho PM^\top)\bs{q}$ and $(M - \rho M^\top) \bs{u}$ are decorrelated.
To compute $\EE  (M-\rho PM^\top)\bs{q} 
  \bs{u}^\top (M - \rho M^\top)^\top$, we use Lemma~\ref{lemma:elliptic-misc}-(i), noticing that the 
cross terms involving $G$ and $\widetilde G$ in the previous characterization 
are zero, and that $\bs{q}^\top \bs{u} = 0$. This leads to 
\begin{eqnarray*}
         \EE  (M-\rho PM^\top)\bs{q} \bs{u}^\top (M - \rho M^\top)^\top &=& \left(\frac{1-\rho^2}{2} (I-\rho P) - \frac{1-\rho^2}{2} (I+\rho P)\right)\bs{u}\bs{q}^\top
         \\  &=&\  -\rho (1-\rho^2)P\bs{u}\bs{q}^\top \ = \ 0 \,,
\end{eqnarray*}
by noticing that $P\bs{u}\bs{q}^\top = 0$.

To obtain that $(M-\rho PM^\top)\bs{q}$ and $U^\top M^\top \bs{u}$ are decorrelated for each column $\bs{u}$ of $U$, we easily notice
that $\EE \left( M-\rho PM^\top\right) \bs{q} \bs{u}^\top M = 
 \text{scalar} \times \bs{u}\bs{q}^\top$ from the structures of $M-\rho M^\top$ and 
$M$ provided above and from Lemma~\ref{lemma:elliptic-misc}-(i). Thus, 
$\EE \left(  M-\rho PM^\top\right) \bs{q} \bs{u}^\top M U = 0$ since
$\bs{q}^\top U = 0$.  
\end{proof}

\begin{proposition}
\label{innov} 
Let $M\sim \Elliptic(n,\rho)$. Let $\bs{q}\in \RR^n$ be a unit-norm deterministic
vector, and let $P$ be a deterministic orthogonal projection matrix on a 
subspace of $\RR^n$ such that $P\bs{q}=0$. Then, 
\[
(M - \rho P M^\top) \bs{q} \sim \cN\left( 0, 
  I - \rho^2 P + \rho \bs{q}\bs{q}^\top  \right) . 
\]
\end{proposition}
\begin{proof}
Using the same principle as in the previous proof, we write 
\[
M - \rho P M^\top \eqlaw 
 \sqrt{\frac{1+\rho}{2}} (I-\rho P)  G + \sqrt{\frac{1-\rho}{2}} (I+\rho P)  
  \widetilde G . 
\]
By Lemma~\ref{lemma:elliptic-misc}-(ii), we then have that 
\begin{equation*}
        (I-\rho P)  G \bs{q}\sim \mathcal{N}\left(0, (I-\rho P)^2 + \bs{q}\bs{q}^\top\right)\quad \text{and}\quad 
        (I+\rho P) \widetilde G \bs{q}\sim \mathcal{N}\left(0,(I+\rho P)^2 - \bs{q}\bs{q}^\top\right).
\end{equation*}
Finally, we get $(M - \rho P M^\top) \bs{q} \sim \cN(0, \Sigma)$ with 
\begin{equation*}
        \Sigma \quad=\quad  \frac{1+\rho}{2} \left((I-\rho P)^2 + \bs{q}\bs{q}^\top\right) + \frac{1-\rho}{2} \left((I+\rho P)^2 - \bs{q}\bs{q}^\top\right)\quad =\quad  I - \rho^2 P + \rho \bs{q}\bs{q}^\top \,,
\end{equation*}
which yields the desired result. 
\end{proof} 


\subsection{Proof of Theorem~\ref{th:main}: notations and some preparation} 
\label{subsec:some-preparation}
We introduce hereafter notations used throughout the proof.

\subsubsection*{Conditioning}
The conditional equality in distribution of two random variables $X$ and $Y$
given a $\sigma$-field $\mathcal{F}$ will be denoted as $X \cond{\cF} Y$. Formally
$$
X \cond{\cF} Y\qquad \text{iff} \qquad 
\EE[\varphi(X) \, | \, \cF] \ =\ 
\EE[\varphi(Y) \, | \, \cF]\quad  (a.s.) 
$$
for every non-negative measurable function $\varphi$.

The conditionnal independence of $X$ and $Y$ givent $\cF$ will be denoted by  $X \indep\hspace{-0.1cm}|_{\mathcal{ F}}\, Y$. Formally,
$$
X \indep\hspace{-0.1cm}|_{\mathcal{ F}}\, Y\qquad \text{iff} \qquad 
\mathbb{E}\left[\left.\varphi(X)\psi(Y)\right| \mathcal{F}\right] \ =\
\mathbb{E}\left[\left.\varphi(X)\right|
\mathcal{F}\right]\mathbb{E}\left[\left.\psi(Y)\right| \mathcal{F}\right]\ (a.s.)
$$
for all non-negative measurable functions $\varphi$
and $\psi$.

The following lemma will be of use later.
\begin{lemma}
\label{lem:ExtendedLawEq}
    Let $\mathcal{F}$, $\mathcal{G}$ be two $\sigma$-fields and $Y, X,\overline{X}$ be random variables. Suppose that (i) $Y$ is $\cF$-measurable, (ii) $X\laweq \overline{X}$ and (iii) $\overline{X} \indep \cF$. Suppose moreover that for some measurable function $\varphi$: 
    \begin{equation*}
           (iv)\ \varphi\left(X,Y\right) \cond{\mathcal{F}} \varphi\left(\overline{X},Y\right)\quad \textrm{and}\quad 
        (v)\ \varphi(X,Y) \indep\hspace{-0.1cm}|_{\mathcal F}\, \mathcal{G} \,.
    \end{equation*}
    Let $\mathcal{H} = \sigma\left(\mathcal{F}\cup \mathcal{G}\right)$, then there exists a random variable $\widetilde{X}$ such that $\widetilde{X}\laweq X$, $\widetilde{X}\indep \mathcal{H}$ and
    \begin{equation*}
        \varphi(X,Y) \cond{\mathcal{H}} \varphi(\widetilde{X},Y).
    \end{equation*}
\end{lemma}

\begin{proof}
    We skip all the integrability issues.
    Let $Z$ be an $\mathcal{H}$-measurable random variable, and $\psi$ any measurable function. Let $\Phi = \psi\circ \varphi$, we have 
    \begin{eqnarray*}
        \mathbb{E}\left[Z\Phi\left(X,Y\right)\right] \ =\   \mathbb{E}\left[\mathbb{E}\left[Z \Phi\left(X,Y\right)\mid \mathcal{F}\right] \right]
        &\stackrel{(a)}=&\mathbb{E}\left[\mathbb{E}\left[Z \mid  \mathcal{F}\right] \mathbb{E}\left[ \Phi\left(X,Y\right) \mid \mathcal{F}\right]\right]\,,\\
        &\stackrel{(b)}=& \mathbb{E}\left[\mathbb{E}\left[Z \mid \mathcal{F}\right] \mathbb{E}\left[\Phi\left(\overline{X},Y\right) \mid \mathcal{F}\right]\right]\,,\\
        &\stackrel{(c)}=& \mathbb{E}\left[\mathbb{E}\left[Z \mid \mathcal{F}\right] \mathbb{E}_{\overline{X}}\left[\Phi\left(\overline{X},Y\right)\right]\right]\,,
    \end{eqnarray*}
    where $(a)$ follows from the fact that $ \varphi(X,Y) \indep\hspace{-0.1cm}|_{\mathcal F}\, \mathcal{G} $ implies $ \varphi(X,Y) \indep\hspace{-0.1cm}|_{\mathcal F}\, \mathcal{H}$, hence $\varphi(X,Y) \indep\hspace{-0.1cm}|_{\mathcal F} Z$ (see for instance \cite[Corollary 6.7]{kallenberg2002foundations}), $(b)$ follows from assumption (ii) and $(c)$ from (iii).

    Consider now a r.v. $\widetilde{X}$ such that $\widetilde{X}\laweq X$ and $\widetilde{X}\indep \mathcal{H}$ then
    $
    \mathbb{E}_{\overline{X}}\,  \Phi(\overline{X}, Y) = \mathbb{E}_{\widetilde{X}}  \Phi(\widetilde{X}, Y)
    $
    and
    $$
          \mathbb{E}\left[Z\Phi\left(X,Y\right)\right] \ =\ \mathbb{E}\left[\mathbb{E}\left(Z \mid \mathcal{F}\right) \mathbb{E}_{\widetilde{X}}\Phi\left(\widetilde{X},Y\right)\right]
            \ =\ \mathbb{E}\left[Z\,\mathbb{E}_{\Tilde{X}}\Phi\left(\Tilde{X},Y\right)\right] 
            \ \stackrel{(d)}=\  \mathbb{E}\left[Z\Phi\left(\widetilde{X},Y\right)\right]\,,
            $$
    where $(d)$ follows from Fubini's theorem. This completes the proof.
 \end{proof}

\subsubsection*{Notational shortcuts}
The following notations, related to the AMP iterations, will be of constant use in the sequel.
\begin{equation*}
\begin{array}{lll}
\bs{q}^k \triangleq h_k\left(\bs{u}_n^k, B_n\right) \phantom{\bigg|}\in \RR^n 
  &\text{for} \ k \geq 0 &\text{and } \bs{q}^{-1}=\bs{0}_n\,,\\
Q_k \triangleq \begin{bmatrix} \bs{q}^0,\cdots,\bs{q}^{k-1} \end{bmatrix} \phantom{\bigg|}
    \in \RR^{n\times k} &\text{for } k \geq 1\,,& \\
U_k \triangleq \begin{bmatrix} \bs{u}^1,\cdots,\bs{u}^k \end{bmatrix} \phantom{\bigg|}
    \in \RR^{n\times k} &\text{for } k \geq 1\,,&\\
d_k \triangleq 
 \left\langle \partial_1 h_k\left(\bs{u}_n^k,B_n\right)\right\rangle_n \in \mathbb{R} \phantom{\bigg|}
 &\text{for} \ k \geq 1, & \text{and}\ d_0 = 0\,.  
\end{array}  
\end{equation*}

Using these notations, Eq.~\eqref{eq:AMP_elliptic} can be written in the following compact form 
\begin{equation}
\label{eq:AMP_compact}
    \bs{u}^{k+1} = A \bs{q}^k - \rho d_k \bs{q}^{k-1}, \quad k \geq 0.
\end{equation}
\subsubsection*{Projections} Denote by $\Pi_{\colspan Q}$ the orthogonal projection matrix on the column
span of the matrix $Q$, and as $\Pi_{{\colspan Q}^\perp}$ the orthogonal 
projection matrix on the orthogonal of the latter subspace. It is well-known that 
$$
\Pi_{\colspan Q_k} = Q_k \left(Q_k^\top Q_k\right)^\dag Q_k^\top\ ,
$$
where $A^\dag$ represents a pseudo-inverse of matrix $A$.

For $k\ge 1$, we introduce the notations
$$
\left\{
\begin{array}{lll}
 P_k &\triangleq \Pi_{\colspan Q_k} &= Q_k \left(Q_k^\top Q_k\right)^\dag 
              Q_k^\top\\
        P_k^\perp &\triangleq \Pi_{{\colspan Q_k}^\perp} &= I_n - P_k           
\end{array}
\right.\qquad \text{and}\qquad \bs{\alpha}^k \ \triangleq\  \left(Q_k^\top Q_k\right)^\dag Q_k^\top \bs{q}^k 
 = \begin{bmatrix} \alpha^k_1 \\ \vdots \\ \alpha^k_k \end{bmatrix} 
 \in \RR^k. 
$$
By convention, $P_0 = 0_{n\times n}$ and $P_0^\perp = I_n$.

\subsubsection*{Filtration} We define the filtration $(\mathcal{F}_k)_{k\geq 1}$ by 
$$\mathcal{F}_k \triangleq
  \sigma\left(B, \bs{u}^0, \bs{u}^1, \cdots, \bs{u}^k\right)\,
  $$

\subsubsection*{Preparing the induction} Recall the definitions of $\bs{q}^k$, $Q_k$, $U_k$, $P_k$ and $\bs{\alpha}^k$ introduced above. The first step of the proof is to establish the following structural result.
\begin{proposition} 
\label{prop:struct}
Let $k,\ell \geq 1$ and define the vectors
$$
\bs{v}^{k,\ell} \triangleq U_k^\top \bs{q}^\ell -  d_\ell Q_k^\top \bs{q}^{\ell-1}
\quad \text{and}\quad \cI_k(A) \triangleq (A - \rho P_k A^\top) P_k^\perp \bq^k\, .
$$
Then $\bs{u}^1 = A\bs{q}^0 = \mathcal{I}_0(A)$ and
\begin{equation*}
        \bu^{k+1} \quad =\quad  \sum_{\ell=1}^{k} \alpha_\ell^k \bu^\ell + 
 \rho Q_k  \left(Q_k^\top Q_k\right)^{\dag} \left(\bs{v}^{k,k} 
- \sum_{\ell=1}^k \alpha_\ell^k \bs{v}^{k,\ell-1}\right) + \cI_k(A)\,, \quad (k\geq 1)\, .
\end{equation*}
\end{proposition} 

Proof of Proposition \ref{prop:struct} is postponed to Appendix~\ref{prop:struct-proof}.

\begin{remark}[mesurability issues]
    Consider the decomposition of $\bu^{k+1}$ in Proposition \ref{prop:struct} then
\[
\sum_{\ell=1}^{k} \alpha_\ell^k \bu^\ell + 
 \rho Q_k  \left(Q_k^\top Q_k\right)^{\dag} \left(\bs{v}^{k,k} 
- \sum_{\ell=1}^k \alpha_\ell^k \bs{v}^{k,\ell-1}\right),
\]
is $\cF_k$--measurable while in general the term $\cI_k(A)$ is
not. The strategy developed by Bolthausen amounts to replace matrix $A$ in $\cI_k(A)$ by some matrix $\widetilde A \indep\hspace{-0.1cm} |_{\cF_k} A$ before proceeding to the induction. This is the goal of next section.
\end{remark}

\subsection{Proof of Theorem~\ref{th:main}: adaptation of Bolthausen conditioning argument} 
\label{subsec:bolthausen}
\begin{proposition}
\label{prop:bolt} 
For $k\geq 1$, there exists a $n\times n$ matrix $\widetilde A$ such that $\widetilde{A} \laweq A$, $\widetilde{A}\indep \cF_k$ and 
$$
(A-\rho  P_k A^\top)P_k^\perp \ \cond{\cF_k} \ 
   (\widetilde{A}-\rho  P_k \widetilde{A}^\top)P_k^\perp\, .
   $$
In particular, 
$$
\cI_k(A) \ \cond{\cF_k}\  \cI_k(\widetilde A)\, .
$$
\end{proposition}
Proposition \ref{prop:bolt}, the proof of which is postponed to the end of Section \ref{subsec:bolthausen}, is a consequence of a more general result stated in Proposition \ref{prop:bolt++}.

Recall that $Q_k=[\bs{q}^0,\cdots, \bs{q}^{k-1}]$ and that
$ P_k = \Pi_{\colspan Q_k} = Q_k \left(Q_k^\top Q_k\right)^\dag Q_k^\top$. Denote by 
\begin{equation}
\label{eq:def-qperp}
\barpp{\bs{q}}{k} \triangleq P_k^{\perp} \bq^k\ .
\end{equation}
Notice that $\barpp{\bq}{k}$ is $\cF_k$-measurable. Let $r_k$ be
the rank of the matrix $Q_k$ - notice that $r_k$ is $\mathcal{F}_{k-1}$--measurable. Let $O_k$ a
$\mathcal{F}_{k-1}$--measurable
$n \times (n-r_k)$ matrix which columns form an
orthonormal basis of $\colspan Q_k^{\perp}$. Such a matrix exists: for instance, consider $\bs{q}^0, \cdots, \bs{q}^{k-1}$ and the deterministic canonical base $(\bs{e}^\ell)_{\ell\in [n]}$ of $\mathbb{R}^n$ and construct by Gram-Schmidt procedure an orthonormal basis of $\mathbb{R}^n$ whose first $r_k$ vectors span $Q_k$. Build $O_k$ out of the remaining $n-r_k$ vectors. In particular,
$$
O_k O_k^\top =P_k^{\perp} \qquad \textrm{and}\qquad O_k^\top O_k = I_{n-r_k}\, .
$$

\begin{proposition}
\label{prop:bolt++}
For every $k\geq 1$, it holds that 
\[
    \begin{bmatrix}
    (A-\rho A^\top)O_k \\
    O_k^\top A^\top O_k
\end{bmatrix} 
  \cond{\cF_k} 
    \begin{bmatrix}
    (\widetilde{A}-\rho \widetilde{A}^\top)O_k \\
    O_k^\top \widetilde{A}^\top O_k
\end{bmatrix}, 
\]
where $\widetilde{A} \laweq A$ and $\widetilde{A}$ is independent of $\cF_k$.
\end{proposition}

\begin{proof}
For $A\in\mathbb{R}^{n\times n}$ and $O\in\mathbb{R}^{n\times n^\prime}$  $(n^\prime\ge 1)$, let
\begin{equation*}
    \mathcal{K}(A,O) = \begin{bmatrix}
    (A-\rho A^\top)O \\
    O^\top A^\top O
\end{bmatrix}\,.
\end{equation*}
We prove the statement by induction on $k\geq 1$ and
begin by proving it for $k=1$. Notice that $Q_1=[\bs{q}^0]$, $O_1$ has dimension $n\times (n-1)$ and $(\bs{q}^0)^\top O_1=0$. Recall that $\mathcal{F}_1=\sigma\left(B,\bs{u}^0, \bs{u}^1 \right)$ and $\bs{u}^1 = A\bs{q}^0$. Taking into account the fact that $A \indep \mathcal{F}_0$ and applying Proposition~\ref{prop:mainIndep} with $P$ (in the proposition) equal to zero, we have:
$$
\mathcal{K}(A, O_1)  \quad \indep\hspace{-0.1cm}|_{\mathcal F_0} \quad \bs{u}^1\,.
$$
Now consider $\overline{A}$ independent from all the considered quantities, then 
$$
\mathcal{K}(A,O) \cond{\mathcal{F}_0} \mathcal{K}(\overline{A},O)
$$
We can now apply Lemma~\ref{lem:ExtendedLawEq} to prove the existence of $\widetilde{A}$ independent of $\mathcal{F}_1 = \sigma\left(\mathcal{F}_0, \{\bs{u}_1\}\right)$ satisfying
\begin{equation*}
    \mathcal{K}(A,O_1) \cond{\mathcal{F}_1} \mathcal{K}(\widetilde{A},O_1)\,.
\end{equation*}
The statement is proved for $k=1$. Suppose now that
\begin{equation}
 \mathcal{K}(A,O_{k-1})
  \cond{\cF_{k-1}} 
   \mathcal{K}(\widetilde{A}, O_{k-1})
\label{proof:HRBolthausen} 
\end{equation}
where $\widetilde{A}$ is independent of $\cF_{k-1}$, and let us prove that this 
equality holds for $k$. Notice that one can assume that $\widetilde{A}$ is independent of $\cF_k$, since this does not change the
conditional distribution of $\mathcal{K}(\widetilde{A}, O_{k-1})$ in \eqref{proof:HRBolthausen}.  

Recall that $\barpp{\bs{q}}{k} = P_k^{\perp} \bq^k$, and observe that the event
$$
E_k \triangleq \{r_{k}=r_{k-1}+1\} \in \cF_{k-1}
$$
coincides with the event $\{ \barpp{\bq}{k-1} \neq 0 \}$.  Define the matrix $W_k$ such that  $W_k \triangleq \begin{bmatrix} \barpp{\bs{q}}{k-1} \, | \, O_k\end{bmatrix}$ on $E_k$ and $W_k \triangleq O_k$ on $E_k^c$. The random matrix $W_k$ is
$\cF_{k-1}$-measurable, so is $O_{k-1}^\top W_k$.
 Moreover, $P_{k-1}^\perp W_k=W_k$. Write
\begin{equation*}
\begin{split}
    \begin{bmatrix}
        (A - \rho A^\top)W_k \\
        W_k^\top A^\top W_k 
    \end{bmatrix} &=
    \begin{bmatrix}
        (A - \rho A^\top)P_{k-1}^\perp W_k \\
        (P_{k-1}^\perp W_k)^\top A^\top P_{k-1}^\perp W_k 
    \end{bmatrix}\\
    &\stackrel{(a)}= 
    \begin{bmatrix}
        (A - \rho A^\top)O_{k-1} \left(O_{k-1}^\top W_k\right) \\
        \left(O_{k-1}^\top W_k\right)^\top O_{k-1}^\top A^\top O_{k-1} 
    \left(O_{k-1}^\top W_k\right)
    \end{bmatrix},
\end{split}
\end{equation*}
where equality $(a)$ holds because $P_{k-1}^\perp W_k=W_k$, then
using the induction hypothesis~\eqref{proof:HRBolthausen}, we get  
\begin{equation*}
    \begin{bmatrix}
        (A - \rho A^\top)W_k \\
        W^\top_k A^\top W_k 
    \end{bmatrix}  \cond{\cF_{k-1}} 
       \begin{bmatrix}
        (\widetilde{A} - \rho \widetilde{A} ^\top)W_k \\
        W_k^\top \widetilde{A}^\top W_k
    \end{bmatrix} , 
\end{equation*}
Substituting by the expression of $W_k$, we have proved  that:
\begin{equation} 
\label{eq:allInfoE} 
    \begin{bmatrix}
        (A - \rho A^\top)\barpp{\bs{q}}{k-1} & (A - \rho A^\top)O_k  \\
        (\barpp{\bs{q}}{k-1})^\top A^\top \barpp{\bs{q}}{k-1} & (\barpp{\bs{q}}{k-1})^\top A^\top O_k \\
        O_k^\top A^\top \barpp{\bs{q}}{k-1} & O_k^\top A^\top O_k 
    \end{bmatrix} \mathds{1}_{E_k}  
    \cond{\cF_{k-1}} 
    \begin{bmatrix}
        (\widetilde{A} - \rho \widetilde{A}^\top)\barpp{\bs{q}}{k-1} & (\widetilde{A} - \rho \widetilde{A}^\top)O_k  \\
        (\barpp{\bs{q}}{k-1})^\top \widetilde{A}^\top \barpp{\bs{q}}{k-1} & (\barpp{\bs{q}}{k-1})^\top \widetilde{A}^\top O_k \\
        O_k^\top \widetilde{A}^\top \barpp{\bs{q}}{k-1} & O_k^\top \widetilde{A}^\top O_k 
    \end{bmatrix}\mathds{1}_{E_k},
\end{equation} 
and
\begin{equation}
\label{eq:allInfoEc} 
    \begin{bmatrix}
    (A-\rho A^\top)O_{k} \\
    O_{k}^\top A^\top O_{k}
\end{bmatrix}\mathds{1}_{E_k^c} 
  \cond{\cF_{k-1}} 
    \begin{bmatrix}
    (\widetilde{A}-\rho \widetilde{A}^\top)O_{k} \\
    O_{k}^\top \widetilde{A}^\top O_{k}
\end{bmatrix}\mathds{1}_{E_k^c}.
\end{equation} 
Recall that $\mathcal{F}_k = \sigma\left(\mathcal{F}_{k-1},\{\bs{u}^{k}\}\right)$, let us study the quantity $\bs{u}^k$.
\begin{equation*}
     \begin{split}
        \bu^k &= A \bq^{k-1} - d_{k-1} \bq^{k-2} \\
        &= AP_{k-1}\bq^{k-1} + AP_{k-1}^{\perp}\bq^{k-1} - d_{k-1} \bq^{k-2}\\
        &= AP_{k-1}\bs{q}^{k-1} 
  + \rho \left(AP_{k-1}\right)^\top \barpp{\bs{q}}{k-1} - d_{k-1} \bq^{k-2} 
   +  \left(A - \rho P_{k-1} A^\top\right)  \barpp{\bs{q}}{k-1}.
    \end{split}
\end{equation*}
We can re-write this expression as 
$$\bu^k = \bs{m}^{k-1} + \bs{z}^{k},$$ 
with 
\begin{equation*}
\bs{z}^{k} \triangleq\left(A - \rho P_{k-1} A^\top\right)  \barpp{\bs{q}}{k-1} \quad \mbox{ and } \quad
 \bs{m}^{k-1} \ \mbox{is }\mathcal{F}_{k-1} \mbox{-measurable}.
\end{equation*}
We now want to prove that 
\begin{equation}
\label{indep:uAndK}
    \bs{u}^k \indep\hspace{-0.1cm}|_{\mathcal{F}_{k-1}}\, \mathcal{K}(A, O_k)  ,
\end{equation}
which is equivalent to 
\begin{equation}
    \bs{z}^k \indep\hspace{-0.1cm}|_{\mathcal{ F}_{k-1}}\, \mathcal{K}(A, O_k),
\end{equation}
which can also be reduced into two smaller problems (see, \emph{e.g.}, \cite[Lemma 7.9.a]{feng2021unifying}),
\begin{eqnarray}
     \bs{z}^k \mathds{1}_{E_k} \indep\hspace{-0.1cm}|_{\mathcal{F}_{k-1}} \mathcal{K}(A, O_k)\mathds{1}_{E_k} , \label{indep:zAndKE}\\
        \bs{z}^k \mathds{1}_{E_k^c} \indep\hspace{-0.1cm}|_{\mathcal{F}_{k-1}} \mathcal{K}(A, O_k)\mathds{1}_{E_k^c} \label{indep:zAndKEc}.
\end{eqnarray}
Let us begin by showing \eqref{indep:zAndKE}. By Equation~\eqref{eq:allInfoE} we have the following equality of joint laws,
\begin{equation*}
    \left(\bs{z}^k \mathds{1}_{E_k}, \ \mathcal{K}(A, O_k)\mathds{1}_{E_k}\right) \cond{\mathcal{F}_{k-1}} \left(\bs{z}^k \mathds{1}_{E_k}, \ \mathcal{K}(\widetilde{A}, O_k)\mathds{1}_{E_k}\right),
\end{equation*}
so in order to show \eqref{indep:zAndKE}, it suffices to show that
\begin{equation*}
    \bs{z}^k \mathds{1}_{E_k} \indep\hspace{-0.1cm}|_{\mathcal{F}_{k-1}} \mathcal{K}(\widetilde{A}, O_k)\mathds{1}_{E_k}.
\end{equation*}
But since $\widetilde{A}$ is independent of $\mathcal{F}_{k-1}$, this is a direct consequence of Proposition~\ref{prop:mainIndep}. Similarly, by noticing that $\bs{z}^k\mathds{1}_{E_k^c}=0$ and using Equation~\eqref{eq:allInfoEc} and the fact that $0$ is independent of any gaussian vector we can prove \eqref{indep:zAndKEc}. Now using \eqref{indep:uAndK} and the induction hypothesis \eqref{proof:HRBolthausen} we can finally apply Lemma~\ref{lem:ExtendedLawEq} to complete the proof.
\end{proof}
We are now in position to prove Proposition \ref{prop:bolt}.
\begin{proof}[Proof of Proposition \ref{prop:bolt}] 
To prove Proposition~\ref{prop:bolt} using Proposition~\ref{prop:bolt++}, we write
$(A-\rho  P_k A^\top)P_k^\perp = (A - \rho A^\top) O_k O_k^\top + 
 \rho O_k O_k^\top A^\top O_k O_k^\top$, and we use Proposition~\ref{prop:bolt++}
along with the following well-known result 
(see, \emph{e.g.}, \cite[Lemma 7.6.c]{feng2021unifying}): If $X$, $X'$, and
$Y$ are random vectors on a probability space, and $\cF$ is a 
$\sigma$--field on this space such that $X \cond{\cF} X'$ and that 
$Y$ is $\cF$--measurable, then, for each measurable function $\varphi$, it
holds that $\varphi(X,Y) \cond{\cF} \varphi(X',Y)$. 
\end{proof} 
Taking advantage of Proposition \ref{prop:bolt}, we can improve Proposition \ref{prop:struct} by replacing ${\mathcal I}_k(A)$ by ${\mathcal I}_k(\widetilde A)$. We also replace some random quantities by their deterministic equivalents. 

Recall the definition of matrices $R^k$ given by the Density Evolution equations \eqref{eq:DE}, we define three related quantities $\sigma_k^2 \in\mathbb{R}^+$, $\stdnew_k\in \mathbb{R}^+$ and $\bar{\bs{\alpha}}^k \in \mathbb{R}^k$, such as 
\begin{eqnarray}
    \sigma_k^2 &=& R_{k,k}\,, \\
    \varnew_{k+1} &=& \sigma_{k+1}^2 - \left(R^{k+1}_{[k],k+1}\right)^\top \left(R^{k}\right)^{-1} \left(R^{k+1}_{[k],k+1}\right)\,,\\
   \bs{\bar{\alpha}}^k &=& \left(R^k\right)^{-1} R^{k+1}_{[k],k+1}\,.
\end{eqnarray}
\begin{remark}
    Notice that $\varnew_{k+1}$ is the Schur complement of $R^k$ in the matrix $R^{k+1}$. One should think of $\stdnew_k$  and $\bar{\bs{\alpha}}^k$ as the deterministic equivalents of $\frac{\| \barpp{\bq}{k} \|}{\sqrt{n}} $ and $\bs\alpha^k$ respectively when $n$ is large.
\end{remark}

\begin{proposition}
\label{prop:ubar}
Using the previous notations we have the following decomposition of $\bs{u}^{k+1}$.
\begin{align*} 
\bu^{k+1} &\cond{\cF_k} \sum_{\ell=1}^{k} \alpha_\ell^k \bu^\ell + 
 \rho Q_k  \left(Q_k^\top Q_k\right)^{\dag} \left(\bs{v}^{k,k} 
- \sum_{\ell=1}^k \alpha_\ell^k \bs{v}^{k,\ell-1}\right) + 
    (\Tilde A - \rho P_k \Tilde A^\top) \barpp{\bq}{k} \\
 &\cond{\cF_k} \bar{\bs{u}}^{k+1} 
       + \Delta^{k+1},
\end{align*} 
where
\begin{equation}\label{eq:def-ubar}
\bar{\bs{u}}^{k+1} \ \triangleq\  \sum_{\ell=1}^{k} \bar\alpha_\ell^k \bu^\ell + \stdnew_k \bs{\xi}^{k+1}\,,
\end{equation}
and $\Delta^{k+1}$ is defined on the events $E_k \triangleq \{\barpp{\bq}{k}\neq \bs{0}\}$ and $E_k^c = \{\barpp{\bq}{k}= \bs{0}\}$ by 
\begin{eqnarray*}
\Delta^{k+1} &\triangleq& 
\sum_{\ell=1}^{k} \left( \alpha^k_\ell - \bar\alpha_\ell^k \right) \bu^\ell 
 + \rho Q_k  \left(Q_k^\top Q_k\right)^{\dag} \left(\bs{v}^{k,k} 
- \sum_{\ell=1}^k \alpha_\ell^k \bs{v}^{k,\ell-1}\right) \\
 && + 
 \left(\frac{\| \barpp{\bq}{k} \|}{\sqrt{n}} - \stdnew_k \right) \bs{\xi}^{k+1} 
 + \left(\sqrt{1-\rho^2} - 1\right)\frac{\| \barpp{\bq}{k} \|}{\sqrt{n}}  P_k \bs{\xi}^{k+1} 
 +\left(\sqrt{1+\rho}-1\right)\frac{\barpp{\bq}{k}(\barpp{\bq}{k})^\top}{\sqrt{n} \| \barpp{\bq}{k} \|} \bs{\xi}^{k+1}\, 
\end{eqnarray*} 
and
\begin{equation*}
    \Delta^{k+1} \triangleq 
 \sum_{\ell=1}^{k} \left( \alpha^k_\ell - \bar\alpha_\ell^k \right) \bu^\ell 
 + \rho Q_k  \left(Q_k^\top Q_k\right)^{\dag} \left(\bs{v}^{k,k} 
- \sum_{\ell=1}^k \alpha_\ell^k \bs{v}^{k,\ell-1}\right),
\end{equation*}
respectively.
\end{proposition} 
A very similar result is obtained in \cite[Section 6.4]{feng2021unifying}. 
\begin{remark}
    The aim of this proposition is to approximate the asymptotic behavior of the distribution of the iterates $\bs{u}^{k+1}$ with the distribution of $\bar{\bs{u}}^{k+1} $ which is easier to handle provided that $\bar\alpha_\ell^k$ and $\stdnew_k$ are deterministic quantities. This is achieved by proving that the difference $\Delta^{k+1}$ is asymptotically negligible.
\end{remark}

\begin{proof} 
We only address the case where $\barpp{\bs{q}}{k}\neq 0$. In the other case, the term $(\Tilde A - \rho P_k \Tilde A^\top) \barpp{\bq}{k}=0$ does not need to be handled.
The starting point is the decomposition of $\bs{u}^{k+1}$ in Proposition~\ref{prop:struct}.
By Proposition~\ref{innov}, the conditional distribution of $(\widetilde{A}-\rho  P_k \widetilde{A}^\top) P_k^\perp \bq^k = 
(\widetilde{A}-\rho  P_k \widetilde{A}^\top) \barpp{\bs{q}}{k}$ given $\cF_k$ is
\[
\cL\left( 
(\widetilde{A}-\rho  P_k \widetilde{A}^\top) \barpp{\bs{q}}{k} \, | \, \cF_k \right) 
= \cN\left( 0, 
\frac{1}{n}\| \barpp{\bq}{k} \|^2 \left( I - \rho^2 P_k + 
  \rho \barpp{\bq}{k} (\barpp{\bq}{k})^\top / \| \barpp{\bq}{k} \|^2 \right)
 \right).
\]
Therefore, letting $\bs{\xi}^{k+1} \sim \cN(0, I_n) $  be independent of $\cF_k$, it holds by Proposition~\ref{prop:bolt} that 
\[
(A-\rho  P_k A^\top)P_k^\perp \bq^k \cond{\cF_k} 
\frac{1}{\sqrt{n}}\| \barpp{\bq}{k} \| \left( I - \rho^2 P_k +\rho \barpp{\bq}{k} (\barpp{\bq}{k})^\top / \| \barpp{\bq}{k} \|^2\right)^{1/2} \bs{\xi}^{k+1}. 
\]
It is clear that we can take 
$$\left( I - \rho^2 P_k +\rho \barpp{\bq}{k} (\barpp{\bq}{k})^\top / \| \barpp{\bq}{k} \|^2 \right)^{1/2} = 
 I + \left(\sqrt{1-\rho^2} - 1\right) P_k + \left(\sqrt{1+\rho}-1\right)\barpp{\bq}{k}(\barpp{\bq}{k})^\top/ \| \barpp{\bq}{k} \|^2\,.$$
The proof of the proposition is completed. 
\end{proof}

\subsection{Proof of Theorem~\ref{th:main}: end of proof}
\label{subsec:end-of-proof-main}
The remainder of the proof of Theorem~\ref{th:main} follows almost verbatim the proof provided by \cite[Section 6]{feng2021unifying}. For completeness, we provide here the important steps of the proof without rigorously justifying all technical details.

We will proceed by induction on $k\geq 1$. Suppose that for any pseudo-Lipschitz function $\varphi$,
\begin{equation}
    \label{eq:HR}\tag{$\mathfrak{H}_k$}
     \frac{1}{n} \sum_{i=1}^n \varphi(B_{i,*}, u^1_i,\cdots, u^k_i) \cconverge \mathbb{E}\left[\varphi(\bs{\Bar{b}}, Z_1,\cdots,Z_k)\right],
    \end{equation}
and we want to prove this same convergence for $k+1$.\\
To do this, fix any pseudo-Lipschitz function $\psi$ and consider the following random variable
\begin{equation*}
    S = \frac{1}{n} \sum_{i=1}^n \psi(B_{i,*}, u^1_i,\cdots, u^k_i, u^{k+1}_i).
\end{equation*}
Recall the definition \eqref{eq:def-ubar} of $\bs{\bar u}^{k+1}$ and write $S$ as follows:
\begin{equation*}
\begin{split}
    S &\cond{\mathcal{F}_k} \frac{1}{n} \sum_{i=1}^n \psi(B_{i,*}, u^1_i,\cdots, u^k_i, \bar{u}^{k+1}_i)\\
    &+ \frac{1}{n}\left(  \sum_{i=1}^n \psi(B_{i,*}, u^1_i,\cdots, u^k_i, u^{k+1}_i) -  \sum_{i=1}^n \psi(B_{i,*}, u^1_i,\cdots, u^k_i, \bar{u}^{k+1}_i) \right), 
\end{split}
\end{equation*}
where $\bs{u}^{k+1}\cond{\mathcal{F}_k}\bar{\bs{u}}^{k+1} + \Delta^{k+1}$. The idea will be then to prove that
\begin{eqnarray}
     S_1 &\triangleq& \frac{1}{n} \sum_{i=1}^n \psi(B_{i,*}, u^1_i,\cdots, u^k_i, \bar{u}^{k+1}_i) \cconverge \mathbb{E}\left[\psi(\bs{\Bar{b}}, Z_1,\cdots,Z_k)\right],\label{eq:convS1}\\
     S_2 &\triangleq& \frac{1}{n}\left(  \sum_{i=1}^n \psi(B_{i,*}, u^1_i,\cdots, u^k_i, u^{k+1}_i) -  \sum_{i=1}^n \psi(B_{i,*}, u^1_i,\cdots, u^k_i, \bar{u}^{k+1}_i) \right) \cconverge 0.\label{eq:convS2}
\end{eqnarray}
Let us begin by proving the convergence in \eqref{eq:convS1}. This proof can be decomposed into two steps, the first step is to prove that the conditional expectation $\mathbb{E}[S_1 \mid \mathcal{F}_k]$ converges to the desired limit, and the second step (which is omitted) is to show that $S_1$ concentrates around its conditional expectation.

In the sequel, we shall rely on the following notation. Let $X$ a random variable and $X\indep Y$. We denote by $\mathbb{E}_X$ the expectation with respect to the distribution of $X$. In particular,
$$
\mathbb{E}_X f(X,Y) = \int f(x,Y)\mathbb{P}_X(dx) = \mathbb{E} ( f(X,Y)\mid Y)\, .
$$
Let us compute the conditional expectation of $S_1$ given $\mathcal{F}_k$.
\begin{equation*}
        \begin{split}
            \mathbb{E}\left[S_1\left.\right|\mathcal{F}_k\right] &= \frac{1}{n}\mathbb{E}\left[\left.\sum_{i=1}^n \psi\left(B_{i,*},u_i^1,\cdots,u_i^k,\sum_{l=1}^k \bar{\alpha}_l^k u_i^l + \stdnew_{k+1} \xi_i^{k+1}\right)\right| \ \mathcal{F}_k\right]\\
            &= \frac{1}{n}\sum_{i=1}^n\mathbb{E}_{\xi_i^{k+1}}\left[ \psi\left(B_{i,*},u_i^1,\cdots,u_i^k,\sum_{l=1}^k \bar{\alpha}_l^k u_i^l + \stdnew_{k+1} \xi_i^{k+1}\right)\right]\\
            & \triangleq \frac{1}{n}\sum_{i=1}^{n}\Psi\left(B_{i,*},u_i^1,\cdots,u_i^{k}\right).\\
        \end{split}
    \end{equation*}
    By \cite[Lemma 7.23]{feng2021unifying}, $\Psi$ is also a pseudo-lipschitz function, thus using the induction hypothesis \eqref{eq:HR} we can write
    \begin{equation*}
        \frac{1}{n}\sum_{i=1}^{n}\Psi\left(B_{i,*},u_i^1,\cdots,u_i^{k}\right) \cconverge\mathbb{E}\left[\Psi\left(Z_1,\cdots,Z_k\right)\right].
    \end{equation*}
    Now, given a random variable $\Tilde{Z}\sim\mathcal{N}(0,1)$ independent of $\mathcal{F}_k$ we can write
    \begin{equation*}
    \mathbb{E}\left[\Psi\left(Z_1, \cdots, Z_{k}\right)\right] = \mathbb{E}\left[\mathbb{E}_{\Tilde{Z}}\left[ \psi\left(Z_1, \cdots, Z_k, \sum_{\ell=1}^k \Bar{\alpha}_\ell^k Z_\ell + \stdnew_{k+1} \Tilde{Z}\right)\right] \right],
    \end{equation*}
     Put $Z_{k+1} \triangleq \sum_{\ell=1}^k \Bar{\alpha}_\ell^k Z_\ell + \stdnew_{k+1} \Tilde{Z}$, and observe that $\left(Z_1,\cdots,Z_{k+1}\right)\sim\mathcal{N}\left(0, R^{k+1}\right)$, then
    \begin{equation*}
    \mathbb{E}\left[\Psi\left(Z_1, \cdots, Z_{k}\right)\right] = \mathbb{E}\left[ \psi\left(Z_1, \cdots, Z_{k+1}\right)\right].
\end{equation*}

Now let us give some proof elements for the convergence in \eqref{eq:convS2}. The idea is to simply use the pseudo-Lipschitz property and bound the term $S_2$ by the distance $\| \bs{u}^k$ - $\bar{\bs{u}}^k \|$. Thus the main ingredient of this proof is to show that 
\begin{equation}
\label{eq:Delta0}
    \frac{1}{\sqrt{n}} \lVert \Delta^{k+1} \rVert \cconverge 0.
\end{equation}
Recall the expression of $\Delta^{k+1}$ in Proposition~\ref{prop:ubar} which can be written as the sum of five terms:
\begin{equation}
\label{eq:defDelta}
    \Delta^{k+1} = \Delta^{(1)}+\Delta^{(2)}+\Delta^{(3)}+\Delta^{(4)}+\Delta^{(5)},
\end{equation}
where 
\begin{equation}
\label{eq:defDeltas}
    \begin{split}
        \Delta^{(1)} &=  \sum_{\ell=1}^{k} \left( \alpha^k_\ell - \bar\alpha_\ell^k \right) \bu^\ell\\
        \Delta^{(2)} &= \rho Q_k  \left(Q_k^\top Q_k\right)^{\dag} \left(\bs{v}^{k,k} 
- \sum_{\ell=1}^k \bs{\alpha}_\ell^k \bs{v}^{k,\ell-1}\right)\\
        \Delta^{(3)} &=  \left(\frac{1}{\sqrt{n}}\| \barpp{\bq}{k} \| - \stdnew_k \right) \xi^{k+1}\\
        \Delta^{(4)} &= \left(\sqrt{1-\rho^2} - 1\right) \frac{1}{\sqrt{n}}\| \barpp{\bq}{k} \|  P_k \xi^{k+1} \\
        \Delta^{(5)} &= \left(\sqrt{1+\rho}-1\right)\frac{\barpp{\bq}{k}(\barpp{\bq}{k})^\top}{\sqrt{n} \| \barpp{\bq}{k} \|} \bs{\xi}^{k+1} .\\
    \end{split}
\end{equation}

In order to prove \eqref{eq:Delta0}, it suffices to show that the normalized norm of each of $\{\Delta^{(j)}, \ j\in [5]\}$ converges to $0$. The key arguments of this proof can be summarized in the following lemma which are direct consequences of the induction hypothesis \eqref{eq:HR}.

\begin{proposition}Let $k\geq 1$ be a fixed integer and suppose that the induction hypothesis \eqref{eq:HR} is satisfied for rank $k$, i.e.
\begin{equation*}
    \forall \varphi \in PL_2\left(\mathbb{R}^{p+k}\right) \ \frac{1}{n}\sum_{i=1}^{n}\varphi\left(B_{i,*}, u_i^1,\cdots, u_i^k\right) \cconverge \mathbb{E}\left[\varphi\left(\bs{\bar{b}},Z_1,\dots,Z_k\right)\right]
\end{equation*}
where $\left(Z_1,\cdots,Z_k\right)$ is a centered gaussian vector of covariance matrix $R^k\in \mathbb{R}^{k\times k}$ which is defined recursively using the Density Evolution equations \eqref{eq:DE}. Then we have the following consequences

\begin{enumerate}[label=(c-1)]
    \item \label{prop:HRC1} For all $j\leq k$, $\frac{1}{n} \lVert \bs{u}^j \rVert^2 \cconverge \mathbb{E}\left[Z_j^2\right] $.
\end{enumerate}
\begin{enumerate}[label=(c-2)]
    \item \label{prop:HRC2} For all $j\leq k$, $\frac{1}{n} \lVert \bs{q}^j \rVert^2 \cconverge \mathbb{E}\left[h_j\left( Z_j,\bs{\bar{b}}\right)^2\right]$.
\end{enumerate}
\begin{enumerate}[label=(c-3)]
    \item \label{prop:HRC3} For all $i,j \leq k$, $\frac{1}{n} \langle \bs{q}^{i-1}, \bs{q}^{j-1} \rangle \cconverge \mathbb{E}\left[h_{i-1}\left(Z_{i-1},\bs{\bar{b}}\right)h_{j-1}\left(Z_{j-1},\bs{\bar{b}}\right)\right] = R^{k}_{i,j} $.
\end{enumerate}
\begin{enumerate}[label=(c-4)]
    \item \label{prop:HRC4} $\frac{1}{n} Q_k^\top Q_k = \left(\frac{1}{n} \langle \bs{q}^{i-1}, \bs{q}^{j-1} \rangle\right)_{1\leq i,j\leq k} \cconverge R^{k} $.
\end{enumerate}
\begin{enumerate}[label=(c-5)]
    \item \label{prop:HRC5} $\bs{\alpha}^k = \left(Q_k^\top Q_k\right)^{\dag} Q_k^\top q^k \cconverge \left(R^k\right)^{-1} R_{[1,k],k+1}^{k+1} = \Bar{\bs{\alpha}}^k$.
\end{enumerate}
\begin{enumerate}[label=(c-6)]
    \item \label{prop:HRC6} $\frac{1}{n}\lVert \barpp{\bs{q}}{k} \rVert^2 \cconverge \stdnew_{k+1}^2$.
\end{enumerate}
\begin{enumerate}[label=(c-7)]
    \item \label{prop:HRC7} For $j\leq k$, $d_j = \frac{1}{n} \sum_{i=1}^n h_j^\prime \left( u_i^j, B_{i,*} \right) \cconverge \mathbb{E}\left[ h_j^\prime \left(Z_j,\bs{\bar{b}}\right) \right] \triangleq \Bar{d}_j$.
\end{enumerate}
\label{prop:consequences}
\end{proposition}

Using this proposition, we can already see that from \ref{prop:HRC1} and \ref{prop:HRC5} we have $\frac{1}{\sqrt{n}}\lVert \Delta^{(1)}\rVert\cconverge 0$ and from \ref{prop:HRC2} we have $\frac{1}{\sqrt{n}}\lVert \Delta^{(3)}\rVert\cconverge 0$.
The quantities $\Delta^{(4)}$ and $\Delta^{(5)}$ are small rank projections of some gaussian vectors and thus their normalized norms converge to $0$. It remains to show that the term 
\begin{equation}
\label{eq:stein}
     \Delta^{(2)}=Q_k  \left(\frac{1}{n}Q_k^\top Q_k\right)^{\dag} \left(\frac{1}{n}\bs{v}^{k,k} 
- \frac{1}{n}\sum_{\ell=1}^k \alpha_\ell^k \bs{v}^{k,\ell-1}\right),
\end{equation}
has a normalized norm that converges to $0$. This can be achieved by showing that
\begin{equation*}
    \frac{1}{n}\bs{v}^{k,\ell} \cconverge 0.
\end{equation*}
The $j$-th row of $\frac{1}{n} \bs{v}^{k,\ell}$ can be written as:
        $$ \frac{1}{n} v^{k,\ell}_j = \frac{1}{n}\langle \bs{u}^{j}, \bs{q}^\ell \rangle - \frac{1}{n} d_\ell \langle \bs{q}^{j-1}, \bs{q}^{\ell-1} \rangle , $$
where $\langle u,v\rangle = u^\top v$.  By Proposition~\ref{prop:consequences} we have:
\begin{itemize}
            \item $d_\ell \cconverge \Bar{d}_\ell$,
            \item $\frac{1}{n} \langle \bs{q}^{j-1}, \bs{q}^{\ell-1} \rangle \cconverge  R^{k}_{j,\ell}$,
            \item $\frac{1}{n} \langle \bs{u}^{j}, \bs{q}^\ell \rangle= \frac{1}{n} \sum_{i=1}^n u_i^j h_\ell(u_i^\ell) \cconverge \mathbb{E}\left(Z_j h_{\ell}(Z_\ell)\right)$.
        \end{itemize}
Using Stein's integration by parts formula and the density evolution equations we get:
        \begin{equation*}
            \begin{split}
                \mathbb{E}\left(Z_j h_{\ell}(Z_\ell)\right) &= \mathbb{E}\left(Z_j Z_\ell\right)\mathbb{E}\left(h'_\ell(Z_\ell)\right) \\
                &= \mathbb{E}\left(h_{j-1}(Z_{j-1}) h_{\ell-1}(Z_{\ell-1})\right)\Bar{d}_\ell\\
                &= \Bar{d}_\ell R^{k}_{j,\ell}.
            \end{split}
        \end{equation*}
Which leads us to the desired result $\frac{1}{n}\bs{v}^{k,\ell}\cconverge 0.$\\
This completes the proof of Theorem~\ref{th:main}.

\subsection{Proof of Corollary \ref{cor:partialAMP}: blockwise convergence of AMP}
\label{subsec:partialAMPProof}
Assume \ref{ass:A1}, \ref{ass:A2prime} and \ref{ass:A3}  - \ref{ass:A5}. To prove Corollary~\ref{cor:partialAMP} for a parameter matrix $B$ of size $n\times p$ we will have to use Theorem~\ref{th:main} for an augmented parameter matrix $B^\prime = [B \ | \ \bs{s}]$ of size $n\times (p+1)$, where $\bs{s}$ represents a ``selector vector". This explains the utility of presenting our main Theorem~\ref{th:main} using multiple parameter vector $(\bs{b}^1,\cdots, \bs{b}^p)$ in \eqref{eq:multipleVector} instead of a single one as what we usually see in the literature.

\noindent Recall the partition defined in \eqref{eq:partition}, and let $\bs{s}\in\mathbb{R}^n$ be a blockwise constant vector with $q$ different values $\tilde{s}_1,\cdots,\tilde{s}_q$ such that
\begin{equation}
\label{eq:selector}
    s_i = \tilde{s}_j \quad \text{if and only if } \quad i\in C_n^{(j)} \quad \text{for} \ i\in [n] \ \text{and} \ j\in [q].
\end{equation}
Put $B^\prime = [\bs{s} \ | \ B]$ and re-write the AMP iteration defined in \eqref{eq:AMP_elliptic} as 
\begin{equation*}
    \bs{u}^{k+1} = A h_k\left(\bs{u}^k, B^\prime\right) - \rho\left\langle \partial_1 h_k\left(\bs{u}^k,B^\prime\right) \right\rangle_n h_{k-1}\left(\bs{u}^{k-1},B^\prime\right),
\end{equation*}
where we abuse the notation for $h_k$ which will depend only on the first $p+1$ coordinates. In this setting, the Density Evolution equations \eqref{eq:DE-updated} remain unchanged. To use Theorem~\ref{th:main} we should verify that the parameter matrix $B^\prime$ satisfies Assumption~\ref{ass:A2}.
\begin{lemma}
\label{lem:assSatisfied}
    Assume that $(\bs{u}^{0},B)$ satisfies Assumption~\ref{ass:A2prime} and let $\bs{s}$ be defined by \eqref{eq:selector}. Then
         $(\bs{u}^0,B^\prime)$ satisfies Assumption~\ref{ass:A2}, i.e. there exists a vector $(\bar{u},\bar{b}_1,\cdots, \bar{b}_p,\bar{s})$ whose distribution belongs to $\mathcal{P}_2\left(\mathbb{R}^{p+2}\right)$ such that
    \begin{equation*}
        \mu^{\bs{u}^0,B^\prime} = \mu^{\bs{u}^0,\bs{b}^1,\cdots,\bs{b}^p,\bs{s}}\xrightarrow[n\to\infty]{\mathcal{P}_r\left(\mathbb{R}^{p+2}\right)} \mathcal{L}\left((\bar{u},\bar{b}_1,\cdots, \bar{b}_p,\bar{s})\right).
    \end{equation*}
    In this case $\mathcal{L}\left(\bar{u},\bar{b}_1,\cdots,\bar{b}_p,\bar{s}\right) = \sum_{j=1}^q c_j \mathcal{L}\left(\bar{u}_j,\bar{b}_{j,1},\cdots,\bar{b}_{j,p}\right) \otimes \delta_{\tilde{s}_j}$, and in particular
    \begin{equation*}
        \bar{s}\sim \sum_{j=1}^q c_j \delta_{\tilde{s}_j} \quad \bar{u}\sim \sum_{j=1}^q c_j \mathcal{L}(\bar{u}_j), \quad \text{and} \quad \bar{b}_\ell\sim \sum_{j=1}^q c_j \mathcal{L}(\bar{b}_{j,\ell})\ \text{for all } \ell \in [p].
    \end{equation*}
    
\end{lemma}
\begin{proof}
    Let $\Upsilon$ be a pseudo-Lipschitz test function, we want to show the existence of $\left(\bar{u},\bar{s},\bar{b}_1,\cdots,\bar{b}_p\right)$ such that 
    \begin{equation*}
        \frac{1}{n}\sum_{i=1}^n \Upsilon\left(u_i^0,b_i^1,\cdots,b_i^p,s_i\right) \cconverge \mathbb{E}\left[\Upsilon\left(\bar{u},\bar{b}_1,\cdots,\bar{b}_p,\bar{s}\right)\right].
    \end{equation*}
The previous sum can be expressed with respect to the partition \eqref{eq:partition} as
\begin{equation*}
\begin{split}
     \frac{1}{n}\sum_{i=1}^n \Upsilon\left(u_i^0,b_i^1,\cdots,b_i^p,s_i\right) &= \sum_{j=1}^q \frac{n_j}{n} \frac{1}{n_j}\sum_{i\in C_n^{(j)}} \Upsilon\left(u_i^0,b_i^1,\cdots,b_i^p,s_i\right)\\
     &= \sum_{j=1}^q \frac{n_j}{n} \frac{1}{n_j}\sum_{i\in C_n^{(j)}} \Upsilon\left(u_i^0,b_i^1,\cdots,b_i^p,\tilde{s}_j\right).\\
\end{split}
\end{equation*}
Thus using Assumption~\ref{ass:A2prime} with the test function $\Upsilon\left(\cdots,\tilde{s}_j\right)$ yields
\begin{equation*}
    \frac{1}{n}\sum_{i=1}^n \Upsilon\left(u_i^0,b_i^1,\cdots,b_i^p,s_i\right)\cconverge \sum_{j=1}^q c_j \mathbb{E}\left[\Upsilon\left(\bar{u}_j,\bar{b}_{j,1},\cdots,\bar{b}_{j,p},\tilde{s}_j\right)\right]\,.
\end{equation*}
Let $\left(\bar{u},\bar{b}_1,\cdots,\bar{b}_p,\bar{s}\right)\sim \sum_{j=1}^q c_j \mathcal{L}\left(\bar{u}_j,\bar{b}_{j,1},\cdots,\bar{b}_{j,p}\right) \otimes \delta_{\tilde{s}_j} $, we have 
\begin{equation*}
    \sum_{j=1}^q c_j \mathbb{E}\left[\Upsilon\left(\bar{u}_j,\bar{b}_{j,1},\cdots,\bar{b}_{j,p},\tilde{s}_j\right)\right] = \mathbb{E}\left[\Upsilon\left(\bar{u},\bar{b}_1,\cdots,\bar{b}_p,\bar{s}\right)\right],
\end{equation*}
    hence the result.
\end{proof}

We can now apply Theorem~\ref{th:main} which gives the following convergence result

\begin{equation}
\label{eq:convBprime}
    \mu^{B^\prime,\bs{u}^1,\cdots, \bs{u}^k} \xrightarrow[n\to\infty]{\mathcal{P}_r\left(\mathbb{R}^{p+k+1}\right)} \mathcal{L}\left(\left(B^\prime,Z_1,\cdots, Z_k\right)\right) \quad \text{(completely)}
\end{equation}

Fix $j$ to be an integer in $[q]$, let $\varphi$ be any pseudo-Lipschitz function on $\mathbb{R}^{p+k}$ and let $\psi$ be a continuous bounded function defined on $\mathbb{R}$ such that:
\begin{equation}
\label{eq:psi}
    \psi(\tilde{s}_\ell) = \begin{cases}
        1 \quad \text{if } \ell = j, \\
        0 \quad \text{otherwise}.
    \end{cases}
\end{equation}
Finally, consider the test function $\phi\in PL_{r}\left(\mathbb{R}^{p+k+1}\right)$ defined as 
\begin{equation*}
    \phi\left(x,\bs{y}\right) = \psi(x)\varphi(\bs{y}) \quad \forall (x,\bs{y})\in \mathbb{R}\times\mathbb{R}^{p+k},
\end{equation*}
and apply the result in \eqref{eq:convBprime} to get 
\begin{equation}
\label{eq:convCons}
    \frac{1}{n}\sum_{i=1}^{n}\phi\left(\tilde{s}_i, b^1_i,\cdots,b^p_i,u^1_i,\cdots,u^k_i\right)\cconverge \mathbb{E}\left[\phi\left(\bar{s},\bar{b}_1,\cdots,\bar{b}_p,Z_1,\cdots,Z_k\right)\right],
\end{equation}
where $\left(\bar{s},\bar{b}_1,\cdots,\bar{b}_p\right)$ is defined as in Lemma~\ref{lem:assSatisfied} and $\left(Z_1,\cdots,Z_k\right)$ is an independent gaussian vector that satisfies the Density evolution equations \eqref{eq:DE-updated} that depend only on $\left(\bar{u},\bar{b}_1,\cdots, \bar{b}_p\right)$.
The structure of $\psi$ in \eqref{eq:psi} implies that the left hand side of \eqref{eq:convCons} can be expressed as  
\begin{equation*}
    \frac{1}{n}\sum_{i=1}^{n}\phi\left(\tilde{s}_i, b^1_i,\cdots,b^p_i,u^1_i,\cdots,u^k_i\right) = \frac{n_j}{n} \frac{1}{n_j} \sum_{i\in C_n^{(j)}}\varphi\left(b^1_i,\cdots,b^p_i,u^1_i,\cdots,u^k_i\right)
\end{equation*}
thus 
\begin{equation*}
    \frac{1}{n_j}\sum_{i\in C_n^{(j)}}\varphi\left(b^1_i,\cdots,b^p_i,u^1_i,\cdots,u^k_i\right)\cconverge \frac{1}{c_j}\mathbb{E}\left[\phi\left(\bar{s},\bar{b}_1,\cdots,\bar{b}_p,Z_1,\cdots,Z_k\right)\right],
\end{equation*}
Now recall the law of the vector $(\bar{s},\bar{\bs{b}})=\left(\bar{s},\bar{b}_1,\cdots,\bar{b}_p\right)$ which is independent of $\bs{Z}=\left(Z_1,\cdots,Z_k\right)$, we have
\begin{equation*}
\begin{split}
        \mathbb{E}\left[\phi\left(\bar{s},\bar{b}_1,\cdots,\bar{b}_p,Z_1,\cdots,Z_k\right)\right] &= \mathbb{E}_{(\bar{s},\bar{\bs{b}})}\left[\mathbb{E}_{\bs{Z}}\left[\phi\left(\bar{s},\bar{b}_1,\cdots,\bar{b}_p,Z_1,\cdots,Z_k\right)\right]\right]\\
        &= \sum_{\ell=1}^q c_\ell \mathbb{E}_{(\bs{\bar{b}},\bs{Z})}\left[\phi\left(\tilde{s}_\ell,\bar{b}_{\ell,1},\cdots,\bar{b}_{\ell,p},Z_1,\cdots,Z_k\right)\right]\\
        &= c_j \mathbb{E}_{(\bs{\bar{b}},\bs{Z})}\left[\varphi\left(\bar{b}_{j,1},\cdots,\bar{b}_{j,p},Z_1,\cdots,Z_k\right)\right]. \\
\end{split}
\end{equation*}
Finally, we get
\begin{equation*}
    \frac{1}{n_j}\sum_{i\in C_n^{(j)}}\varphi\left(b^1_i,\cdots,b^p_i,u^1_i,\cdots,u^k_i\right) \cconverge \mathbb{E}_{(\bs{\bar{b}},\bs{Z})}\left[\varphi\left(\bar{b}_{j,1},\cdots,\bar{b}_{j,p},Z_1,\cdots,Z_k\right)\right]\,,
\end{equation*}
which ends the proof.

\section{Remaining proofs of Section \ref{sec:ecology}}
\label{sec:proofs-ecology}

\subsection{Proof of Theorem \ref{th:main-LV}:  AMP algorithm to describe the LV equilibrium's statistics}
\label{subsec:AMP-for-LV}
Notice that the existence of an equilibrium $\bs{x}_n^\star$ is granted by Proposition \ref{prop:existence-equilibrium} under the condition $\kappa\ge \sqrt{2(1+\rho)}$. The rest of the proof follows very closely \cite[Section 3.3]{akjouj2023equilibria} with Theorems \ref{th:main} and \ref{th:main-LV} to handle the elliptic case. We shall often drop subscript $n$ to lighten the notations.

\subsubsection*{Related Linear Complementarity Problem}
For a LV system, it is well-known that the equilibrium satisfies a non-linear optimization problem called Linear Complementarity Problem (LCP), see \cite{takeuchi1996global}. The LCP problem $\LCP\left(I-\Sigma, -\bs{r}\right)$ with parameters matrices $I,\Sigma\in \mathbb{R}^{n\times n}$ and vector $\bs{r}\in \mathbb{R}^n$ consists in finding a vector $\bs{x}^\star$ satisfying
\begin{equation}\label{eq:LCP-first-sight}
\left\{ 
\begin{array}{lcl}
x^\star_i &\ge& 0\,,\\
x^\star_i \left( r_i - \left[(I-\Sigma)\bs{x}^\star\right]_i\right)&=&0\,,\\
r_i - \left[(I-\Sigma)\bs{x}^\star\right]_i &\le& 0\,,
\end{array}\right.
\qquad \textrm{for}\ i\in [n]\,.
\end{equation}
If such a vector exists, we write $\bs{x}^\star\in \LCP\left(I-\Sigma, -\bs{r}\right)$. The first condition follows from the fact that $\bs{x}_n(t)$ is always (component-wise) positive for a LV system, the second condition simply express the nullity of the derivative at equilibrium. The last condition is a necessary condition (see for instance \cite[Th. 3.2.5]{takeuchi1996global}) for Lyapunov stability and has also an ecological interpretation of non-invasibility (see \cite[Section 3(a)]{akjouj2022complex}). In \cite[Prop. 4]{akjouj2023equilibria}, it is proved that the solution of a $\LCP\left(I-\Sigma, -\bs{r}\right)$ equivalently satisfies a fixed point equation in the sense that:
\begin{equation}\label{eq:LCP-equiv}
\bs{z} =\Sigma \bs{z}_+ +\bs{r}_n \quad \Leftrightarrow\quad \bs{z}_+ \in \LCP\left(I-\Sigma, -\bs{r}\right)\, .
\end{equation}
Otherwise stated, in case of uniqueness, $\bs{z}_+=\bs{x}^\star$.

\subsubsection*{An AMP algorithm} Let $(\delta, \sigma, \gamma)$ be the solution of System \eqref{eq:sys}. 
Define the activation function $h_k$ by:
$$
h_k(u,a)=\frac{(u+a)_+}{\delta} \quad \textrm{for}\quad k\ge 0\qquad \textrm{with} \qquad \partial_1 h_k (u,a) = \frac{\bs{1}_{(u+a>0)}}{\delta}\, ,
$$
and consider the following AMP algorithm 
\begin{equation}
    \label{eq:AMP-ecology}
    \bs{u}^{k+1} = \frac{A_n}{\delta}\left(\bs{u}^k+\bs{a}\right)_{+} - \rho\frac{\langle \bs{1}_{(\bs{u}^k+\bs{a}>0)} \rangle_n  \left(\bs{u}^{k-1}+\bs{a}\right)_+}{\delta^2}\ ,
\end{equation}    
where     
$$
    \begin{cases}
    \bs{u}^0&=\bs{1}_n\\
    \bs{a}&=\left(  1+\rho \frac{\gamma}{\delta^2}\right) \bs{r}=\frac \kappa \delta \bs{r}
    \end{cases}\ .
$$
Notice that by Assumption $\mu^{\bs{a}}\xrightarrow[n\to\infty]{\mathcal{P}_2\left(\RR\right)} \bar{a}$ where $\bar{a}=\left( 1+\rho \frac{\gamma}{\delta^2}\right)\bar{r}$. We can easily check that Assumptions\ref{ass:A1}-\ref{ass:A5} are satisfied and hence can apply Theorem \ref{th:main}. If one is only interested in the limiting law of $\mu^{\bs{u}^k}$, the DE equations write
\begin{equation}\label{eq:DE-simple-case}
\begin{cases}
    \theta_1^2 &= \frac 1{\delta^2} \mathbb{E}(1+\bar{a})_+^2    \\
    \theta_{k+1}^2&= \frac 1{\delta^2}\mathbb{E}(\theta_k \bar{Z} + \bar{a})_+^2
\end{cases}\ ,
\end{equation}
where $\bar{Z}\sim{\mathcal N}(0,1)$ is independent from $\bar{a}$. Theorem \ref{th:main} yields
$$
\mu^{\bs{u}^k} \xrightarrow[n\to\infty]{\mathcal{P}_2\left(\RR\right)} Z_k \sim{\mathcal N}(0,\theta_k^2)\, .
$$
Departing from \eqref{eq:AMP-ecology}, we establish a perturbed LCP with respect to \eqref{eq:LCP-equiv}. Denote 
$$
\bs{\xi}^k = \bs{u}^k +\bs{a}\qquad \textrm{and}
\qquad \gamma^k = \langle \bs{1}_{(\bs{u}^k+\bs{a}>0)} \rangle_n\, .
$$
Taking advantage of the definition of $\bs{a}$ and the relations between $\delta, \sigma$ and $\gamma$ from \eqref{eq:sys}, easy (but lengthy) computations yield
$$
\bs{\xi}_+^k - \frac{\bs{\xi}_-^k}{1+\rho \frac{\gamma}{\delta^2}} = \Sigma\bs{\xi}^k_+ +\bs{r} 
+\frac{\bs{\varepsilon}^k}{1+\rho \frac{\gamma}{\delta^2}}\ ,
$$
where 
$$
\bs{\varepsilon}^k = \frac{\rho \gamma}{\delta^2} \left( \bs{\xi}_+^k -  \bs{\xi}_+^{k-1}\right)
+ \frac{\rho}{\delta^2}( \gamma - \gamma^k) \bs{\xi}_+^{k-1} + \left( \bs{\xi}^k - \bs{\xi}^{k+1}\right)\, .
$$
Defining $\bs{z}^k = \bs{\xi}_+^k - \frac{\bs{\xi}_-^k}{1+\rho \frac{\gamma}{\delta^2}}$ and 
${\bs{\tilde \varepsilon}^k}= \frac{\bs{\varepsilon}^k}{1+\rho \frac{\gamma}{\delta^2}}$, we remark that $\bs{z}^k_+=\bs{\xi}_+^k$ and end up with the fixed-point equation
$
\bs{z}^k =\Sigma_n \bs{z}^k_+ + \bs{r} +\bs{\tilde \varepsilon}^k
$.
Otherwise stated 
\begin{equation}\label{eq:LCP-approx}
\bs{z}_+^k \in \LCP (I-\Sigma, - \bs{r} - \bs{\tilde \varepsilon}^k)\, .
\end{equation}
We first focus on the asymptotic distribution of $\mu^{\bs{z}_+^k}$. Setting $$\sigma_k = \frac{\delta}{\kappa} \theta_k$$ and noticing that function $(u+a)_+$ is Lipschitz, we obtain by Theorem \ref{th:main} that
$$
\mu^{\bs{z}_+^k}\quad \xrightarrow[n\to\infty]{\mathcal{P}_2\left(\RR\right)} \quad {\mathcal L}\left( \left( 1+\frac{\rho\gamma}{\delta^2}\right) (\sigma_k \bar{Z} +\bar{r})_+\right)\, .
$$
Replacing $\theta_k$ by $\sigma_k$ in the DE equations yields the equation
$$
\sigma_{k+1}^2 =\frac 1{\delta^2} \mathbb{E} (\sigma_k \bar{Z} +\bar{r})_+^2 
$$
which by \cite[Lemma 2]{akjouj2023equilibria} yields that $\sigma_k \xrightarrow[k\to\infty]{} \sigma$, the solution of \eqref{eq:sys-sigma}. Hence
\begin{equation}\label{eq:distribution-limite-LCP}
{\mathcal L}\left( \left( 1+\frac{\rho\gamma}{\delta^2}\right) (\sigma_k \bar{Z} +\bar{r})_+\right)
\quad \xrightarrow[k\to\infty]{\mathcal{P}_2\left(\RR\right)}\quad 
{\mathcal L}\left( \left( 1+\frac{\rho\gamma}{\delta^2}\right) (\sigma \bar{Z} +\bar{r})_+\right)\, .
\end{equation}
The arguments to establish convergence \eqref{cvg-muN} in Theorem \ref{th:main-LV} from \eqref{eq:LCP-equiv}, \eqref{eq:LCP-approx} and \eqref{eq:distribution-limite-LCP} follow exactly those in \cite[Section 3.4]{akjouj2023equilibria} and are thus omitted.

\subsection{Proof of Corollary \ref{coro:chaos-prop-I}}
\label{subsec:proof-chaos-prop-I}
We rely on the following result, see Sznitman \cite{sznitman89}:

\begin{proposition}{Chaos propagation, \cite{sznitman89}}
\label{prop:sznitman}
    Let $\left(X_1,\cdots,X_n\right)$ be a random vector of law $P_n$ and let $\mu_n$ be its empirical measure $\mu_n = \frac 1n \sum_{i=1}^n\delta_{X_i}\in\mathcal{P}(\mathbb{R})$. Assume the following
    \begin{enumerate}[label=(i)]
    \item\label{ass-chaos-I:(i)}There exists an probability measure $\mu\in\mathcal{P}(\mathbb{R})$, such that the random probability measure $\mu_n$ converges to $\mu$ in law.
\end{enumerate}
\begin{enumerate}[label=(ii)]
    \item\label{ass-chaos-I:(ii)}The vector $\left(X_1,\cdots,X_n\right)$ is exchangeable, that is for each permutation $\sigma\in \mathcal{S}_n$:
            $$\left(X_{\sigma(1)},\cdots,X_{\sigma(n)}\right) \laweq \left(X_1,\cdots,X_n\right)\, .$$
\end{enumerate}

    Under these assumptions, the probability distribution $P_n$ is $\mu$-chaotic, that is for each fixed integer $K$ we have
    \begin{eqnarray*}
        \left(X_1,\cdots,X_K\right) \xrightarrow[n\to\infty]{\mathcal L} \mu^{\otimes K}.
    \end{eqnarray*}
    \end{proposition}
    Corollary~\ref{coro:chaos-prop-I} is a direct consequence of Proposition~\ref{prop:sznitman}, we only need to verify that $\bs{x}^\star = \left(x_1^\star,\cdots,x_n^\star\right)$ satisfies the two assumptions.

\subsubsection*{Proof of Assumption~\ref{ass-chaos-I:(i)}} By Theorem~\ref{th:main-LV} we have the following convergence
    \begin{equation*} 
\mu^{\bs{x}^\star} \ \xrightarrow[n\to\infty]{\mathcal{P}_2(\RR)} \
 \pi:=\mathcal{L}\left( \left( 1 + \rho \gamma/\delta^2 \right) 
  \left( \sigma \bar Z + \bar r \right)_+ \right) \quad (completely)\ ,
\end{equation*}
which implies the convergence in probability of $\mu^{\bs{x}^\star}$ and thus the convergence in law.

\subsubsection*{Proof of Assumption~\ref{ass-chaos-I:(ii)}} We now prove that $\bs{x}^\star$ is exchangeable. Let the permutation $\sigma\in {\mathcal S}_n$ be fixed and $P_{\sigma}\in \mathbb{R}^{n\times n}$ its associated permutation matrix. We introduce the set 
$$
{\mathcal E}(A)=\left\{ \frac{\| A\|}{\kappa}<1 \right\}\, .
$$
Suppose now that $\frac{\| A\|}{\kappa}<1$ then $\bs{x}^\star = \bs{x}^\star =\bs{z}_+$ where $\bs{z} = \Sigma \bs{z}_+ +\bs{r}$. 
The function 
$$
\bs{y}\mapsto \Sigma \bs{y}_+ +\bs{r}
$$ 
is Lipschitz with Lipschitz parameter $\| \Sigma\|<1$ hence $\bs{z}=\lim_p \bs{z}(p)$ where $\bs{z}(p)$ is defined by:
$$
\begin{cases}
    \ \bs{z}(0)=0\,,\\
    \ \bs{z}(p+1) = \Sigma \bs{z}_+(p) +\bs{r}\, .
\end{cases}
$$
Consider the following notations: 
$$
A^\sigma = P_\sigma^{-1} A P_{\sigma}\,,\quad \bs{y}^\sigma = P_\sigma \bs{y} \quad \textrm{for any}\ \bs{y}\in \mathbb{R}^n\, .
$$
At first, we consider the $\bs{z}(p)$'s regardless of the condition $\frac{\|A\|}{\kappa}<1$ and prove by induction that 
\begin{equation}
\label{eq:induction-exchangeability}
\forall\ p\ge 0\,,\quad {\mathcal L}(\bs{z}(p), A, \bs{r} ) = {\mathcal L}(\bs{z}^\sigma(p),A^\sigma, \bs{r}^\sigma )\,.
\end{equation}
Since $P_{\sigma}$ is orthogonal, the invariance property of elliptic matrices implies that $A^{\sigma} \laweq A$. 
For $p=1$, $\bs{z}(1)= \bs{r}$ and ${\mathcal L}(\bs{r}^\sigma, A^\sigma) ={\mathcal L}(\bs{r}, A)$ (recall that $\bs{r}\indep A$) hence the induction property. Now
$$
{\mathcal L}(\bs{z}(p+1), A, \bs{r}) \ =\  {\mathcal L}(A\bs{z}_+(p) +\bs{r}, A, \bs{r})\ \stackrel{(a)}=\  {\mathcal L}(A^\sigma \bs{z}^\sigma_+(p)+\bs{r}, A, \bs{r}^\sigma)\ =\ {\mathcal L}(\bs{z}^\sigma(p+1), A^\sigma, \bs{r}^\sigma)\,,
$$
where $(a)$ follows from the induction hypothesis. Eq.\eqref{eq:induction-exchangeability} is proved.

We can now transfer the exchangeability to $\bs{z}$ conditionnally on ${\mathcal E}(A)$. Notice that ${\mathcal E}(A)={\mathcal E}(A^\sigma)$ and take any bounded continuous test function $\Phi$, then
$$
\mathbb{E} \Phi( \bs{z}(p))\mathbf{1}_{{\mathcal E}(A)} = \mathbb{E} \Phi( \bs{z}^\sigma(p))\mathbf{1}_{{\mathcal E}(A^\sigma)}\, .
$$
Letting $p\to\infty$ yields 
\begin{equation}\label{eq:conditionnal-exchangeability}
{\mathcal L}(\bs{z} \mid {\mathcal E}(A)) = {\mathcal L}(\bs{z}^\sigma \mid {\mathcal E}(A))
\end{equation}
Now if $\frac{\| A\|}{\kappa}>1$ then $\bs{x}^\star=0$ and $P_\sigma \bs{x}^\star=0$. Combining this remark with \eqref{eq:conditionnal-exchangeability} finally yields that ${\mathcal L}(\bs{x}^\star) = {\mathcal L}(P_\sigma \bs{x}^\star)$. The exchangeability of $\bs{x}^\star$ is proved.

\subsection{Proof of Theorem~\ref{th:chaos-prop-II}} 
\label{subsec:proof-chaos-prop-II}
We follow the same strategy as in the proof of Corollary~\ref{coro:chaos-prop-I} except that we need the blockwise form of the AMP theorem (see Corollary~\ref{cor:partialAMP}) and a the generalized version of Proposition~\ref{prop:sznitman} stated hereafter. 

\begin{proposition}\label{prop:sznitman_generalized}
Consider the partition $\left\{C_n^{(j)}\right\}_{j\in[q]}$ defined by \eqref{eq:nj}-\eqref{eq:Cj}. Let $X=\left(X_1, \cdots, X_n\right)$ be a random vector of law $P_n$ and let $\mu_n^{(1)},\cdots, \mu_n^{(q)} \in \mathcal{P}\left(\mathbb{R}\right)$ be the empirical measures of the $q$ blocks of $X$ respectively, i.e.
\begin{eqnarray*}
    \mu_n^{(j)} \triangleq \frac{1}{n_j} \sum_{i\in C_n^{(j)}} \delta_{X_i} \quad \text{for all } j\in[q].
\end{eqnarray*}

Assume the following
\begin{enumerate}[label=(i)]
    \item\label{ass-chaos-II:(i)} There exist $q$ probability measures $\mu_1,\cdots, \mu_q \in \mathcal{P}\left(\mathbb{R}\right)$ such that the random vector $$\left(\mu_n^{(1)},\cdots, \mu_n^{(q)}\right)$$ converges in law to the vector $\bs{\mu}=\left(\mu_1,\cdots,\mu_q\right)$ in the product space $\mathcal{P}\left(\mathbb{R}\right)^q$.
\end{enumerate}

\begin{enumerate}[label=(ii)]
    \item\label{ass-chaos-II:(ii)}The vector $X=\left(X_1,\cdots,X_n\right)$ is blockwise exchangeable (see Definition~\ref{def:blockEx}), that is for each permutations $\bs{\sigma} = (\sigma_1,\cdots, \sigma_q)\in \mathcal{S}_{n_1}\times\cdots\times \mathcal{S}_{n_q}$ we have the following $X^{\bs{\sigma}} \laweq X$.
\end{enumerate}

Under these assumptions, the probability distribution $P_n$ is $\bs{\mu}$-chaotic, that is for each fixed $q$-uplet of integers $\left(k_1, \cdots, k_q\right)\in [n_1]\times\cdots,\times [n_q]$ we have 
\begin{equation*}
    X_{[k_1,\cdots, k_q]} \xrightarrow[n\to\infty]{\mathcal{L}} \prod_{j=1}^q \mu_j^{\otimes k_j},
\end{equation*}
where $X_{[k_1,\cdots,k_q]}$ is the $k_1+\cdots+k_q$-dimensional vector obtained by a concatenation the vectors $\left(X_{i}\right)_{i\in\mathcal{K}_n^{(1)}},\cdots,\left(X_{i}\right)_{i\in\mathcal{K}_n^{(q)}}$ such that $\mathcal{K}_n^{(j)}$ is the subset of the $k_j$ first elements of $C_n^{(j)}$.
\end{proposition}

Proof of Proposition \ref{prop:sznitman_generalized} is postoned to Appendix \ref{app:Sznitman}. 

Theorem~\ref{th:chaos-prop-II} is a direct consequence of Proposition~\ref{prop:sznitman_generalized} and we only need to check that vector $\bs{x}^\star = \left(x_1^\star,\cdots, x_n^\star\right)$ satisfies assumptions \ref{ass-chaos-II:(i)} and \ref{ass-chaos-II:(ii)}. 

\subsubsection*{Proof of Assumption~\ref{ass-chaos-II:(i)}}
Consider the empirical measure of the coordinates of the $j$-th block of $\bs{x}^\star$
    \begin{equation*}
        \mu^{(j)}_n = \frac{1}{n_j} \sum_{i\in C_n^{(j)}}\delta_{x_i^\star}.
    \end{equation*}
    By blockwise AMP we already have 
     \begin{equation*}
        \mu^{(j)}_n \xrightarrow[n\to\infty]{\mathcal{P}_2\left(\mathbb{R}\right)} \pi_j \quad (completely)\,.
    \end{equation*}
    If we endow the space of probability measures with the following distance 
    \begin{equation*}
        \tilde{d}(\mu,\nu) :=  \sup_{\psi} \left|\int_{\mathbb{R}}\psi d\mu - \int_{\mathbb{R}}\psi d\nu\right|,
    \end{equation*}
    then we also have convergence in probability of the sequence of measures $\left(\mu_n^{(j)}\right)_n$ (considered as random variables living in the space of probability measures) to $\pi_j$, the underlying distance is $\tilde{d}$,
    \begin{equation*}
        \mu_n^{(j)} \stackrel{\mathbb{P}, \tilde{d}}{\xrightarrow[n \to +\infty]{}} \pi_j,
    \end{equation*}
    the convergence in probability of the components $\left(\mu^{(j)}\right)_{j\in [q]}$ implies the joint convergence in probability
    \begin{equation*}
        \left(\mu_n^{(1)},\cdots,\mu_n^{(q)}\right) \stackrel{\mathbb{P}, d}{\xrightarrow[n \to +\infty]{}} \left(\mu_1,\cdots,\mu_q\right)
    \end{equation*}
    where the underlying distance $d$ is a distance on the product of $q$ probability measure spaces.
    Finally, and in particular we have convergence in law.

   \subsubsection*{Proof of Assumption~\ref{ass-chaos-II:(ii)}} The proof that $\bs{x}^\star = \left(x_1^\star,\cdots, x_n^\star\right)$ is blockwise exchangeable  closely follows the lines of the proof of exchangeability presented in Section~\ref{subsec:proof-chaos-prop-I} and is thus omitted.


\begin{thebibliography}{AHMN23}

\bibitem[ABC{\etalchar{+}}22]{akjouj2022complex}
I.~Akjouj, M.~Barbier, M.~Clenet, W.~Hachem, M.~Ma{\"\i}da, F.~Massol,
  J.~Najim, and V-C. Tran.
\newblock Complex systems in ecology: A guided tour with large lotka-volterra
  models and random matrices.
\newblock {\em arXiv preprint arXiv:2212.06136}, 2022.

\bibitem[AHMN23]{akjouj2023equilibria}
I.~Akjouj, W.~Hachem, M.~Maïda, and J.~Najim.
\newblock Equilibria of large random lotka-volterra systems with vanishing
  species: a mathematical approach, 2023.

\bibitem[AT12]{allesina2012stability}
Stefano Allesina and Si~Tang.
\newblock Stability criteria for complex ecosystems.
\newblock {\em Nature}, 483(7388):205--208, 2012.

\bibitem[AT15]{allesina2015stability}
S.~Allesina and S.~Tang.
\newblock The stability--complexity relationship at age 40: a random matrix
  perspective.
\newblock {\em Population Ecology}, 57(1):63--75, 2015.

\bibitem[BK17]{bar-krz-it17}
Jean Barbier and Florent Krzakala.
\newblock {A}pproximate {M}essage-{P}assing decoder and capacity achieving
  sparse superposition codes.
\newblock {\em IEEE Transactions on Information Theory}, 63(8):4894--4927,
  2017.

\bibitem[BLM15]{Bayati_2015}
M.~Bayati, M.~Lelarge, and M.~Montanari.
\newblock Universality in polytope phase transitions and message passing
  algorithms.
\newblock {\em The Annals of Applied Probability}, 25(2), 4 2015.

\bibitem[BM11]{bayati2011dynamics}
M.~Bayati and A.~Montanari.
\newblock The dynamics of message passing on dense graphs, with applications to
  compressed sensing.
\newblock {\em IEEE Transactions on Information Theory}, 57(2):764--785, 2011.

\bibitem[Bol14]{bolthausen2014iterative}
E.~Bolthausen.
\newblock An iterative construction of solutions of the tap equations for the
  sherrington--kirkpatrick model.
\newblock {\em Communications in Mathematical Physics}, 325(1):333--366, 2014.

\bibitem[Bun17]{Bunin}
Guy Bunin.
\newblock Ecological communities with lotka-volterra dynamics.
\newblock {\em Phys. Rev. E}, 95:042414, 2017.

\bibitem[CEFN22]{clenet2022equilibrium}
M.~Clenet, H.~El~Ferchichi, and J.~Najim.
\newblock Equilibrium in a large lotka--volterra system with pairwise
  correlated interactions.
\newblock {\em Stochastic Processes and their Applications}, 153:423--444,
  2022.

\bibitem[DAM17]{desh-abb-mon-17}
Yash Deshpande, Emmanuel Abbe, and Andrea Montanari.
\newblock Asymptotic mutual information for the balanced binary stochastic
  block model.
\newblock {\em Information and Inference: A Journal of the IMA}, 6(2):125--170,
  2017.

\bibitem[DMLS23]{dudeja2023universality}
R.~Dudeja, Y.~M.~Lu, and S.~Sen.
\newblock Universality of approximate message passing with semirandom matrices.
\newblock {\em The Annals of Probability}, 51(5):1616--1683, 2023.

\bibitem[DMM09]{Donoho_2009}
David~L. Donoho, Arian Maleki, and Andrea Montanari.
\newblock Message-passing algorithms for compressed sensing.
\newblock {\em Proceedings of the National Academy of Sciences},
  106(45):18914--18919, 2009.

\bibitem[Fan21]{fan2021approximate}
Zhou Fan.
\newblock Approximate message passing algorithms for rotationally invariant
  matrices, 2021.

\bibitem[FVRS21]{feng2021unifying}
Oliver~Y. Feng, Ramji Venkataramanan, Cynthia Rush, and Richard~J. Samworth.
\newblock A unifying tutorial on approximate message passing, 2021.

\bibitem[Gal18]{Galla_2018}
Tobias Galla.
\newblock Dynamically evolved community size and stability of random
  lotka-volterra ecosystems(a).
\newblock {\em Europhysics Letters}, 123(4):48004, 2018.

\bibitem[GH82]{geman1982chaos}
S.~Geman and C-R. Hwang.
\newblock A chaos hypothesis for some large systems of random equations.
\newblock {\em Zeitschrift f{\"u}r Wahrscheinlichkeitstheorie und Verwandte
  Gebiete}, 60(3):291--314, 1982.

\bibitem[Gir86]{girko1986elliptic}
V.~Girko.
\newblock Elliptic law.
\newblock {\em Theory of Probability \& Its Applications}, 30(4):677--690,
  1986.

\bibitem[JM13]{javanmard2013state}
A.~Javanmard and A.~Montanari.
\newblock State evolution for general approximate message passing algorithms,
  with applications to spatial coupling.
\newblock {\em Information and Inference: A Journal of the IMA}, 2(2):115--144,
  2013.

\bibitem[Kal02]{kallenberg2002foundations}
O.~Kallenberg.
\newblock {\em Foundations of modern probability}.
\newblock Probability and its Applications (New York). Springer-Verlag, New
  York, second edition, 2002.

\bibitem[LHM{\etalchar{+}}04]{leibold2004metacommunity}
M.~Leibold, M.~Holyoak, N.~Mouquet, P.~Amarasekare, J.~Chase, M.~Hoopes,
  R.~Holt, J.~Shurin, R.~Law, D.~Tilman, M.~loreau, and A.~Gonzalez.
\newblock The metacommunity concept: a framework for multi-scale community
  ecology.
\newblock {\em Ecology letters}, 7(7):601--613, 2004.

\bibitem[LM19]{lel-mio-ptrf19}
Marc Lelarge and L{\'e}o Miolane.
\newblock Fundamental limits of symmetric low-rank matrix estimation.
\newblock {\em Probability Theory and Related Fields}, 173:859--929, 2019.

\bibitem[May72]{may1972will}
R.~M. May.
\newblock Will a large complex system be stable?
\newblock {\em Nature}, 238(5364):413--414, 1972.

\bibitem[Mon21]{mon-siam-19}
Andrea Montanari.
\newblock Optimization of the {S}herrington--{K}irkpatrick hamiltonian.
\newblock {\em SIAM Journal on Computing}, 0(0):FOCS19--1--FOCS19--38, 2021.

\bibitem[Nau12]{naumov2012elliptic}
Alexey Naumov.
\newblock Elliptic law for real random matrices, 2012.

\bibitem[OR14]{10.1214/EJP.v19-3057}
Sean O'Rourke and David Renfrew.
\newblock {Low rank perturbations of large elliptic random matrices}.
\newblock {\em Electronic Journal of Probability}, 19(none):1 -- 65, 2014.

\bibitem[RGV17]{rush-gre-ven-17}
Cynthia Rush, Adam Greig, and Ramji Venkataramanan.
\newblock Capacity-achieving sparse superposition codes via {A}pproximate
  {M}essage {P}assing decoding.
\newblock {\em IEEE Transactions on Information Theory}, 63(3):1476--1500,
  2017.

\bibitem[RSF18]{rangan2018vector}
Sundeep Rangan, Philip Schniter, and Alyson~K. Fletcher.
\newblock Vector approximate message passing, 2018.

\bibitem[Szn91]{sznitman89}
A-S. Sznitman.
\newblock Topics in propagation of chaos.
\newblock In Paul-Louis Hennequin, editor, {\em Ecole d'Et{\'e} de
  Probabilit{\'e}s de Saint-Flour XIX --- 1989}, pages 165--251, Berlin,
  Heidelberg, 1991. Springer Berlin Heidelberg.

\bibitem[Tak96]{takeuchi1996global}
Y.~Takeuchi.
\newblock {\em Global dynamical properties of Lotka-Volterra systems}.
\newblock World Scientific, 1996.

\end{thebibliography}

\newcommand{\etalchar}[1]{$^{#1}$}

\begin{appendix}

\section{Proof of Lemma~\ref{lemma:main-system}}\label{app:proof-system}
Recall the assumptions of Lemma \ref{lemma:main-system} and system \eqref{eq:sys-delta}-\eqref{eq:sys-gamma}. 

Notice that \cite[Section 3.2]{akjouj2023equilibria} give the proof of existence and uniqueness of a solution $(\delta,\sigma, \gamma)$ in the case where $\rho=1$ and a careful reading indicates that the proof remains true for $\rho\ge 0$. We thus assume $\rho\in [-1,0)$ in the sequel. 

We remind a few facts from \cite[Lemma 2]{akjouj2023equilibria} which concern Eq. \eqref{eq:sys-sigma}-\eqref{eq:sys-gamma} and remain true for $\rho<0$:
\begin{itemize}
    \item[-] For every $\delta>1/\sqrt{2}$, the solution $\sigma(\delta)$ of \eqref{eq:sys-sigma} exists,
    \item[-] For every $\delta>1/\sqrt{2}$, inequality $\gamma(\delta) < \delta^2$ holds true,
    \item[-] Function  $\delta\mapsto \sigma(\delta)$ is a decreasing function such that $ \lim_{\delta\to (1/\sqrt{2})^+} \sigma(\delta)=+\infty$.
\end{itemize}
Given $\delta$ we obtain $\sigma(\delta)$ by solving \eqref{eq:sys-sigma} and $\gamma(\delta)$ by \eqref{eq:sys-gamma}. In order to find a solution $\delta$ which satisfies \eqref{eq:sys-delta}, we study function 
$$
h(\delta) = \delta +\rho\frac{\gamma(\delta)}{\delta} 
$$
and prove that there exists a unique $\delta^\star\in (1/\sqrt{2}, \infty)$ such as $h(\delta^\star)=\kappa$. The existence of $\delta^\star$ is easy to establish. In fact, function $\delta\mapsto h(\delta)$ is continuous,
$$\lim_{\delta\to\infty} h\left(\delta\right) = +\infty > \kappa\qquad \text{and} \qquad \lim_{\delta\to (1/\sqrt{2})^+} h(\delta) = \frac{1}{\sqrt{2}} + \rho \frac{1/2}{1/\sqrt{2}} = \frac{1+\rho}{\sqrt{2}}<\kappa\,,$$
which implies the existence of some $\delta^\star$ satisfying $h(\delta^\star)=\kappa$. Now in order to prove the uniqueness of $\delta^\star$, it suffices to prove that $h$ is strictly increasing. For $\rho <0$ and $\delta>1/\sqrt{2}$ we have 
\begin{equation}
        h^\prime(\delta) \quad =\quad  1 + \rho \frac{\gamma^\prime(\delta)}{\delta} - \rho \frac{\gamma(\delta)}{\delta^2}
        \quad \geq\quad  1+ \rho \frac{\gamma^\prime(\delta)}{\delta} 
         \quad \geq\quad  1- \frac{\gamma^\prime(\delta)}{\delta}\,.
\end{equation}
In order to prove that $h$ is increasing, it suffices to establish the following inequality:
\begin{lemma}
\label{lem:gammap}
Let $\delta \in (1/\sqrt{2},+\infty)$ then $\gamma^\prime\left(\delta\right) < \delta$.
\end{lemma}
Before proceeding to the proof of Lemma \ref{lem:gammap}, we introduce auxiliary functions that will simplify the forthcoming computations. Let 
$$
    Q(x) = \int_{x}^{+\infty} \frac{e^{-t^2/2}}{\sqrt{2\pi}}dt\,,\qquad f(x) = \left(1+x^2\right)Q(x) - x\times \frac{e^{-x^2/2}}{\sqrt{2\pi}}\,,
$$
the derivative of which are given by:
$$
        Q'(x) = -  \frac{e^{-x^2/2}}{\sqrt{2\pi}}\,,\qquad f'(x) = 2\left(xQ(x)-  \frac{e^{-x^2/2}}{\sqrt{2\pi}}\right)\,.
$$
We consider the change of variable
$$
x(\delta)= - \frac 1{\sigma(\delta)} <0
$$
and rewrite the system of equations \eqref{eq:sys-delta}-\eqref{eq:sys-gamma} using $Q$, $f$ and $x(\delta)$:
\begin{eqnarray}
        \kappa &=& \delta + \rho \frac{\gamma\left(\delta\right)} \delta\,, \label{eq:t-delta}\\
        \delta^2 &=& \mathbb{E}\left[f(\bar{r}x\left(\delta\right))\right]\,, \label{eq:t-sigma}\\
        \gamma\left(\delta\right) &=& \mathbb{E}\left[Q(\bar{r}x\left(\delta\right))\right]\,, \label{eq:t-gamma}
\end{eqnarray}
with $(\delta,x,\gamma)\in (1/\sqrt{2},\infty) \times (-\infty, 0) \times (0,1)$. In order to obtain Eq.\eqref{eq:t-sigma}, we start from $\eqref{eq:sys-gamma}$ which is rewritten as 
\begin{eqnarray}
    \delta^2&=&\frac 1{\sigma^2} \mathbb{E} \int_{-\frac{\bar{r}}{\sigma}}^{\infty} (\sigma^2 t^2+2\sigma t\bar{r} + \bar{r}^2)\frac{e^{-\frac{t^2}2}}{\sqrt{2\pi}}\, dt \quad =\quad  \mathbb{E} \int_{x\bar{r}}^{\infty} (t^2-2t\bar{r}x + (x\bar{r})^2)\frac{e^{-\frac{t^2}2}}{\sqrt{2\pi}}\, dt\ ,\nonumber \\
    &=& \mathbb{E} (x\bar{r})^2 Q(x\bar{r}) + \mathbb{E} \int_{x\bar{r}}^{\infty} (t^2-2xt\bar{r})\frac{e^{-\frac{t^2}2}}{\sqrt{2\pi}}\, dt \,,\nonumber \\
    &=&  \mathbb{E} (1+(x\bar{r})^2) Q(x\bar{r}) - \mathbb{E} x\bar{r} \frac{e^{-\frac{(x\bar{r})^2}2}}{\sqrt{2\pi}}\ ,
\label{eq:delta-square}
\end{eqnarray}
where the last equality follows from an integration by parts.

With notations $Q$, $f$, $x(\delta)$ and equation \eqref{eq:t-gamma} at hand, notice that 
$\gamma$ depends on $\delta$ via $x$. The next lemma expresses the derivatives of $x$ and $\gamma$.
\begin{lemma}
\label{lem:derivX}
    Let $\delta > 1/\sqrt{2}$ then $x$ and $\gamma$'s derivatives with respect to $\delta$ write:
$$
         x^\prime\left(\delta\right) \ =\  \frac{\delta x\left(\delta\right)}{\delta^2 - \mathbb{E}\left[Q\left(\bar{r}x\left(\delta\right)\right)\right]}\,,\qquad 
    \gamma^\prime\left(\delta\right) \ =\ - \frac{\delta x\left(\delta\right)}{\delta^2 - \mathbb{E}\left[Q\left(\bar{r}x\left(\delta\right)\right)\right]}\times \frac{\mathbb{E}\left[\bar{r}e^{-\left(\bar{r}x\left(\delta\right)\right)^2/2}\right]}{\sqrt{2\pi}}\,.
$$
\end{lemma}
Momentarily assuming Lemma \ref{lem:derivX}, we are now in position to prove Lemma \ref{lem:gammap}.
\begin{proof}[Proof of Lemma \ref{lem:gammap}]
For $\delta>\frac{1}{\sqrt{2}}$, using Lemma~\ref{lem:derivX} and Eq. \eqref{eq:delta-square} we get
\begin{equation*}
    \frac{\gamma^\prime(\delta)}{\delta} = \frac{1 - \mathbb{E}\left[\left(1+\left(\bar{r}x\left(\delta\right)\right)^2\right)\frac{Q\left(\bar{r}x\left(\delta\right)\right)}{\delta^2}\right]}{1-\mathbb{E}\left[\frac{Q\left(\bar{r}x\left(\delta\right)\right)}{\delta^2}\right]}.
\end{equation*}
The inequality $\gamma^\prime\left(\delta\right) <  \delta$ is equivalent to
\begin{equation*}
    1 - \mathbb{E}\left[\left(1+\left(\bar{r}x\left(\delta\right)\right)^2\right)\frac{Q\left(\bar{r}x\left(\delta\right)\right)}{\delta^2}\right] < 1-\mathbb{E}\left[\frac{Q\left(\bar{r}x\left(\delta\right)\right)}{\delta^2}\right]\qquad \Leftrightarrow\qquad 
    0 < \mathbb{E}\left[\left(\bar{r}x\left(\delta\right)\right)^2\frac{Q\left(\bar{r}x\left(\delta\right)\right)}{\delta^2}\right]\,,
\end{equation*}
the last inequality being true because $\mathcal{L}(\bar{r}) \neq \delta_0$.
\end{proof}
\begin{proof}[Proof of Lemma~\ref{lem:derivX}]
    By differentiating \eqref{eq:t-sigma} with respect to $\delta$ we get:
    \begin{equation*}
                   2 \delta = x^\prime\left(\delta\right) \mathbb{E}\left[\bar{r}f^\prime\left(\bar{r}x\left(\delta\right)\right)\right]\,,
    \end{equation*}
    from which we extract $x'(\delta)$. Using the explicit form of $f'$ yields:
    \begin{equation*}
            x^\prime\left(\delta\right) = \frac{\delta}{\mathbb{E}\left[\bar{r}^2 x\left(\delta\right)Q\left(\bar{r}x\left(\delta\right)\right)\right]-\mathbb{E}\left[\left(\bar{r}/\sqrt{2\pi}\right)e^{-\left(\bar{r}x\left(\delta\right)\right)^2/2}\right]}\,. 
    \end{equation*}
From \eqref{eq:t-sigma} and \eqref{eq:delta-square}, we get:
$$
     \mathbb{E}\left[\frac{\bar{r}e^{-\left(\bar{r}x \left(\delta\right)\right)^2 /2}}{\sqrt{2\pi}}\right] = \mathbb{E}\left[\left(\frac{1}{x\left(\delta\right)}+\bar{r}^2 x \left(\delta\right)\right)Q\left(\bar{r}x \left(\delta\right)\right)\right] - \frac{\delta^2}{x\left(\delta\right)},
$$
thus
\begin{equation*}
            x^\prime\left(\delta\right) \ =\ \frac{\delta}{\mathbb{E}\left[\bar{r}^2 x\left(\delta\right)Q\left(\bar{r}x\left(\delta\right)\right)\right] -\mathbb{E}\left[\left(\frac{1}{x\left(\delta\right)}+\bar{r}^2 x \left(\delta\right)\right)Q\left(\bar{r}x \left(\delta\right)\right)\right] + \frac{\delta^2}{x\left(\delta\right)}} \ =\  \frac{\delta x\left(\delta\right)}{\delta^2 - \mathbb{E}\left[Q\left( \bar{r}x\left(\delta\right)\right)\right]}\,.
\end{equation*}
\end{proof}
Proof of Lemma \ref{lemma:main-system} is completed.

\section{Generalized propagation of chaos }\label{app:Sznitman}
In this section we prove the generalized version of Sznitman's propagation of chaos result presented in section~\ref{subsec:proof-chaos-prop-II} (see Proposition~\ref{prop:sznitman_generalized}) to cover the case of blockwise structured random vectors. 

We provide the following proof of Proposition~\ref{prop:sznitman_generalized} that follows the same ideas of \cite{sznitman89}.

\subsection*{Proof of Proposition~\ref{prop:sznitman_generalized}}
    We want to prove that for any test functions $\varphi^{(1)},\cdots, \varphi^{(q)}$ where $\varphi^{(j)}$ is the tensor product of $k_j$ test functions $\varphi^{(j)}_{1},\cdots,\varphi^{(j)}_{k_j}$ we have the following limit
    \begin{equation}
    \label{eq:lawconv}\mathbb{E}\left[\varphi^{(1)}\left(X^{(1)}\right)\times\cdots\times\varphi^{(q)}\left(X^{(q)}\right)\right] \xrightarrow[n\to\infty]{} \prod_{i=1}^{k_1}\mathbb{E}\left[\varphi_i^{(1)}\left(\Tilde{X}_1\right)\right] \times \cdots \times \prod_{i=1}^{k_q}\mathbb{E}\left[\varphi_i^{(q)}\left(\Tilde{X}_q\right)\right],
    \end{equation}
    where $(\Tilde{X}_1, \cdots, \Tilde{X}_q)\in\mathbb{R}^q$ is a random vector having the law $\bs{\mu}$. The term $\varphi^{(j)}\left(X^{(j)}\right)$ is equal to the following product
    \begin{equation*}
        \varphi^{(j)}\left(X^{(j)}\right) = \varphi^{(j)}_1\left(X^{(j)}_1\right)\times\cdots\times \varphi^{(j)}_{k_j}\left(X^{(j)}_{k_j}\right).
    \end{equation*}
In order to prove \eqref{eq:lawconv}, we will consider the following intermediate term 
\begin{equation*}
    B_n = \mathbb{E}\left[\prod_{k=1}^{k_1}\left(\frac{1}{n_1}\sum_{i=1}^{n_1}\varphi^{(1)}_{k}\left(X^{(1)}_i\right)\right)\times \cdots \times \prod_{k=1}^{k_q}\left(\frac{1}{n_q}\sum_{i=1}^{n_q}\varphi^{(q)}_{k}\left(X^{(q)}_i\right)\right)\right],
\end{equation*}
we will also denote the left hand side and the right hand side of \eqref{eq:lawconv} by $A_n$ and $C_n$ respectively. So it is sufficient to prove that $A_n - B_n \xrightarrow[n\to\infty]{} 0$ and $B_n - C_n \xrightarrow[n\to\infty]{} 0$.

\begin{lemma}
    We have $B_n - C_n \xrightarrow[n\to\infty]{}0$.
\end{lemma}
\begin{proof}
    This immediate by Assumption~\ref{ass-chaos-II:(i)}. In fact, consider the continuous bounded test function $F:\mathcal{P}(\mathbb{R})^q \rightarrow \mathbb{R}$ defined for any $q$-uplet of probability measures $(\nu_1,\cdots,\nu_q)$ by
    \begin{equation*}
        F\left(\nu_1,\cdots,\nu_q\right) = \prod_{i_1=1}^{k_1}\left(\int_{\mathbb{R}}\varphi_{i_1}^{(1)}d\nu_1\right)\times\cdots\times \prod_{i_q=1}^{k_q}\left(\int_{\mathbb{R}}\varphi_{i_q}^{(q)}d\nu_q\right).
    \end{equation*}
Using Assumption~\ref{ass-chaos-II:(i)} yields to the desired result.
\end{proof}

\begin{lemma}
    We have $A_n-B_n \xrightarrow[n\to\infty]{}0$
\end{lemma}

\begin{proof}
    We have
    \begin{equation*}
        A_n = \mathbb{E}\left[\prod_{i_1=1}^{k_1}\varphi^{(1)}_{i_1}\left(X^{(1)}_{i_1}\right)\times\cdots\times\prod_{i_q=1}^{k_q}\varphi^{(q)}_{i_q}\left(X^{(q)}_{i_q}\right)\right].
    \end{equation*}
    By Assumption~\ref{ass-chaos-II:(ii)} we can re-write $A_n$ as follows
    \begin{equation*}
        A_n = \frac{1}{\prod_{j=1}^q \left(n_j!\right)}\sum_{\sigma_1\in\mathfrak{S}_1,\cdots,\sigma_q\in\mathfrak{S}_q} \mathbb{E}\left[\prod_{i_1=1}^{k_1}\varphi^{(1)}_{i_1}\left(X^{(1)}_{\sigma_1(i_1)}\right)\times\cdots\times\prod_{i_q=1}^{k_q}\varphi^{(q)}_{i_q}\left(X^{(q)}_{\sigma_q(i_q)}\right)\right],
    \end{equation*}
Now observe that in the $j$-th factor in the formula above, the product is taken only over the subset $[k_j]$ of $[n_j]$, thus we consider the equivalence relation $\sim_{j}$ defined on $\mathfrak{S}_j$ by 
\begin{equation*}
    \forall \sigma,\nu \in \mathfrak{S}_{n_j} \quad \left[\sigma\sim_{j}\nu \right] \ \Leftrightarrow \ \left[ \sigma(i)=\nu(i), \ \forall i \in [k_j]\right],
\end{equation*}
in oder words, we will identify permutations that agree on the set $[k_j]$, this means that for all $j\in [q]$, if $\sigma_j \sim_j \nu_j $ then 
\begin{equation*}
\prod_{i_1=1}^{k_1}\varphi^{(1)}_{i_1}\left(X^{(1)}_{\sigma_1(i_1)}\right)\times\cdots\times\prod_{i_q=1}^{k_q}\varphi^{(q)}_{i_q}\left(X^{(q)}_{\sigma_q(i_q)}\right) =\prod_{i_1=1}^{k_1}\varphi^{(1)}_{i_1}\left(X^{(1)}_{\nu_1(i_1)}\right)\times\cdots\times\prod_{i_q=1}^{k_q}\varphi^{(q)}_{i_q}\left(X^{(q)}_{\nu_q(i_q)}\right).
\end{equation*}
Now consider the quotient group of $\mathfrak{S}_{n_j}$ with respect to $\sim_j$ denoted as $\tilde{\mathfrak{S}}_j $, then we have the following:
\begin{equation*}
        A_n = \left(\prod_{j=1}^q \frac{(n_j-k_j)!}{n_j!}\right)\sum_{\sigma_1\in\tilde{\mathfrak{S}}_1,\cdots,\sigma_q\in\tilde{\mathfrak{S}}_q} \mathbb{E}\left[\prod_{i_1=1}^{k_1}\varphi^{(1)}_{i_1}\left(X^{(1)}_{\sigma_1(i_1)}\right)\times\cdots\times\prod_{i_q=1}^{k_q}\varphi^{(q)}_{i_q}\left(X^{(q)}_{\sigma_q(i_q)}\right)\right].
    \end{equation*}
Now let us analyse the term $B_n$. If we develop the products inside the $\mathbb{E}$ symbol, we get
\begin{equation*}
\begin{split}
     B_n &= \mathbb{E}\left[\frac{1}{n_1^{k_1}}\sum_{i^{(1)}_1,\cdots,i^{(1)}_{k_1}=1}^{n_1} \prod_{\ell=1}^{k_1}\varphi^{(1)}_\ell\left(X^{(1)}_{i^{(1)}_\ell}\right) \times\cdots\times \frac{1}{n_q^{k_q}}\sum_{i^{(q)}_1,\cdots,i^{(q)}_{k_q}=1}^{n_q} \prod_{\ell=1}^{k_q}\varphi^{(q)}_\ell\left(X^{(q)}_{i^{(q)}_\ell}\right)  \right]\\
     &= \frac{1}{\prod_{j=1}^{q} n_j^{k_j}}\sum_{\substack{i_1^{(1)},\cdots,i_{k_1}^{(1)} \in [n_1]\\ \vdots\\ i_1^{(q)},\cdots,i_{k_q}^{(q)}\in [n_q]}}\mathbb{E}\left[ \prod_{\ell=1}^{k_1}\varphi^{(1)}_\ell\left(X^{(1)}_{i^{(1)}_\ell}\right) \times\cdots\times  \prod_{\ell=1}^{k_q}\varphi^{(q)}_\ell\left(X^{(q)}_{i^{(q)}_\ell}\right)  \right] \\
\end{split}
\end{equation*}
For all $j\in [q]$, let $S_j$ be the subset of $[n_j]^{k_j}$ of $k_j$-uplets with different elements. The sum in $B_n$ is over the grid $[n_1]^{k_1}\times\cdots\times [n_q]^{k_q}$, we will decompose it into a sum over $S_1\times \cdots\times S_q$ and a sum over the complementary set $[n_1]^{k_1}\times\cdots\times [n_q]^{k_q} \setminus S_1\times\cdots\times S_q$, let us denote these two sums by $B^\prime_n$ and $B^{\prime\prime}_n$ respectively. We will show that 
\begin{equation*}
    A_n - B^\prime_n \xrightarrow[n\to\infty]{} 0 \quad \text{and}\quad B^{\prime\prime}_n \xrightarrow[n\to\infty]{} 0.
\end{equation*}
Let us first prove that $A_n-B_n^\prime \xrightarrow[n\to\infty]{}0$. We have 
\begin{equation*}
    B^\prime_n = \frac{1}{\prod_{j=1}^{q} n_j^{k_j}} \sum_{\substack{(i_1^{(1)},\cdots,i_{k_1}^{(1)}) \in S_1\\ \vdots\\ (i_1^{(q)},\cdots,i_{k_q}^{(q)})\in S_q}}\mathbb{E}\left[ \prod_{\ell=1}^{k_1}\varphi^{(1)}_\ell\left(X^{(1)}_{i^{(1)}_\ell}\right) \times\cdots\times  \prod_{\ell=1}^{k_q}\varphi^{(q)}_\ell\left(X^{(q)}_{i^{(q)}_\ell}\right)  \right] .
\end{equation*}
Observe now that we can identify each $k_j$-uplet $(i_1^{(j)},\cdots,i_{k_j}^{(j)})$ of $S_j$ with an element $\sigma_j\in \tilde{\mathfrak{S}}_j$, and this is because by definition of the set $S_j$ the indices $\{i_1^{(j)},\cdots,i_{k_j}^{(j)}\}$ are different, i.e. one can construct $\sigma_j\in \tilde{\mathfrak{S}}_j$ such that 
\begin{equation*}
    \sigma_j\left(\ell\right) = i_\ell^{(j)} \quad \forall \ell\in [k_j],
\end{equation*}
this essentially means that we can index the sum in $B_n^\prime$ using permutations,
\begin{equation*}
    B^\prime_n = \frac{1}{\prod_{j=1}^{q} n_j^{k_j}} \sum_{\sigma_1\in\tilde{\mathfrak{S}}_1,\cdots,\sigma_q\in\tilde{\mathfrak{S}}_q}\mathbb{E}\left[ \prod_{\ell=1}^{k_1}\varphi^{(1)}_\ell\left(X^{(1)}_{\sigma_1(\ell)}\right) \times\cdots\times  \prod_{\ell=1}^{k_q}\varphi^{(q)}_\ell\left(X^{(q)}_{\sigma_q(\ell)}\right)  \right].
\end{equation*}
Except the multiplicative factor, the terms $A_n$ and $B^\prime_n$ are exactly the same, so let us consider the difference 
\begin{equation*}
\begin{split}
    A_n-B_n^\prime =& \left(\prod_{j=1}^{q}\frac{(n_j-k_j)!}{n_j!} - \frac{1}{\prod_{j=1}^{q}n_j^{k_j}}\right)\\
&\times\sum_{\sigma_1\in\tilde{\mathfrak{S}}_1,\cdots,\sigma_q\in\tilde{\mathfrak{S}}_q}\mathbb{E}\left[ \prod_{\ell=1}^{k_1}\varphi^{(1)}_\ell\left(X^{(1)}_{\sigma_1(\ell)}\right) \times\cdots\times  \prod_{\ell=1}^{k_q}\varphi^{(q)}_\ell\left(X^{(q)}_{\sigma_q(\ell)}\right)  \right],
\end{split}
\end{equation*}
denote $M$ an upper bound of the test functions $\varphi_{1}^{(1)},\cdots,\varphi_{k_q}^{(q)}$ and notice that the set $\tilde{\mathfrak{S}}_1\times\cdots\times\tilde{\mathfrak{S}}_q$
is of cardinal $\prod_{j=1}^{q}n_j! / \prod_{j=1}^{q}(n_j-k_j)!$, thus we get
\begin{equation*}
\begin{split}
     \left|A_n-B_n^\prime\right| &\leq \left(\prod_{j=1}^{q}\frac{(n_j-k_j)!}{n_j!} - \frac{1}{\prod_{j=1}^{q}n_j^{k_j}}\right)\left(\prod_{j=1}^{q}\frac{n_j!}{(n_j-k_j)!}\right) M^{k_1+\cdots+k_q}\\
     &= \left(1 - \prod_{j=1}^{q}\frac{n_j!}{n_j^{k_j}(n_j-k_j)!}\right) M^{k_1+\cdots+k_q}. \\
\end{split}
\end{equation*}
Recall that when $n\to\infty$ we also have $n_j(n)\to\infty$ because $c_j>0$. As $k_j\leq n_j$ is a constant integer we have 
$$ \frac{n_j!}{n_j^{k_j}(n_j-k_j)!} \xrightarrow[n\to\infty]{}1\quad \forall j \in [q], $$
finally we get $A_n-B_n^\prime \xrightarrow[n\to\infty]{}0$. It remains to show that $B_n^{\prime\prime}\xrightarrow[n\to\infty]{}0$. Now, $B_n^{\prime\prime}$ is a sum over the complementary part of $S_1\times\cdots \times S_q$ in $[n_1]^{k_1}\times\cdots\times [n_q]^{k_q}$, this complementary subset is of size
$$ n_1^{k_1}\times\cdots\times n_q^{k_q} - \frac{n_1!\times \cdots\times n_q!}{(n_1 - k_1)!\times \cdots\times (n_q - k_q)!}, $$
if we upper bound the test functions by $M$ again, we obtain the following inequality
\begin{equation*}
    \left|B_n^{\prime\prime}\right| \leq \left(1 - \prod_{j=1}^{q}\frac{n_j!}{n_j^{k_j}(n_j-k_j)!}\right) M^{k_1+\cdots+k_q},
\end{equation*}
thus $B_n^{\prime\prime}\xrightarrow[n\to\infty]{}0$.
\end{proof}
 In summary, we have proved that each term of the following expression: $$A_n-C_n = (A_n-B_n^\prime) + B_n^{\prime\prime} + (B_n-C_n)$$ converges to $0$, which ends the proof of Proposition~\ref{prop:sznitman_generalized}.

\section{Proof of Proposition \ref{prop:struct}}
\label{prop:struct-proof}
Define the matrix 
\begin{align*} 
D_k &\triangleq \begin{bmatrix} d_0 \\ & d_1 \\ & & \ddots \\ & & & d_{k-1} \end{bmatrix} \ \text{for } k\geq 1,
\end{align*}  
Recall the compact form of the AMP iteration~\eqref{eq:AMP_compact}
\[
    \bs{u}^{k+1} = A \bs{q}^k - \rho d_k \bs{q}^{k-1}, \quad k \geq 0. 
\] 
Considering the iterates $\bs u^1, \ldots, \bs u^k$ provided by the previous equation and using the notations we just introduced, it is easy to see that
\[
 A Q_k = U_k + \rho  \begin{bmatrix}0 & Q_{k-1}\end{bmatrix} \ D_k, \quad k \geq 1 , 
\]
where $Q_0$ is by convention the empty matrix, (i.e. $ \begin{bmatrix}
    0 & Q_0
\end{bmatrix}=  [0]\in \mathbb{R}^{1\times 1}$).
For $k \geq 1$, we now write 
\begin{align*} 
\bu^{k+1} &= A \bq^k - \rho d_k \bq^{k-1} 
          = A (P_k + P_k^\perp) \bq^k - \rho d_k \bq^{k-1}  \\
 &= A P_k \bq^k + A P_k^\perp \bq^k + \rho (A P_k)^\top P_k^\perp \bq^k 
    - \rho (A P_k)^\top P_k^\perp \bq^k - \rho d_k \bq^{k-1}  \\
 &= A P_k \bq^k + \rho (A P_k)^\top P_k^\perp \bq^k - \rho d_k \bq^{k-1} 
  + \cI_k(A) .
\end{align*} 
The first term at the right hand side can be rewritten as 
\[
A P_k \bq^k = A Q_k \left(Q_k^\top Q_k\right)^{\dag} Q_k^\top \bq^k 
  = \left(U_k + \rho   \begin{bmatrix}0 & Q_{k-1}\end{bmatrix} \ D_k \right) 
    \balpha^k,  
\]
and the next term is rewritten as 
\begin{align*} 
\rho (A P_k)^\top P_k^\perp \bq^k &= 
 \rho  \left(A Q_k \left(Q_k^\top Q_k\right)^{\dag} Q_k^\top\right)^\top 
   P_k^\bot \bq^k \\
 &=\rho  \left(\left(U_k + \rho   \begin{bmatrix}0 & Q_{k-1}\end{bmatrix} \ D_k\right)  
   \left(Q_k^\top Q_k\right)^{\dag} Q_k^\top\right)^\top P_k^\bot \bq^k \\
 &= \rho Q_k  \left(Q_k^\top Q_k\right)^{\dag}\left( U_k^\top P_k^{\perp} 
   \bs{q}^k + \rho D_k^\top \begin{bmatrix}0 & Q_{k-1}\end{bmatrix}^\top P_k^\perp  
    \bq^k \right) \\
 &= \rho Q_k  \left(Q_k^\top Q_k\right)^{\dag} U_k^\top P_k^{\perp} \bs{q}^k, 
\end{align*} 
since $P_k^\perp Q_{k-1} = 0$. This gives 
\begin{equation}
\label{uk+1} 
\bu^{k+1} = \sum_{l=1}^k \alpha^k_l \bu^l 
 + \rho \begin{bmatrix}0 & Q_{k-1}\end{bmatrix}\ D_k \balpha^k 
 + \rho Q_k  \left(Q_k^\top Q_k\right)^{\dag} U_k^\top P_k^{\perp} \bs{q}^k 
  - \rho d_k P_k \bq^{k-1} + \cI_k(A), 
\end{equation} 
since $\bq^{k-1} \in \colspan P_k$. Developing in turn the three middle terms 
at the right hand side of this equation, and observing that 
$P_k \bq^l = \bq^l$ for each $l \leq k-1$, we obtain 
\[
 \rho  \begin{bmatrix}0 & Q_{k-1}\end{bmatrix} \ D_k \balpha^k = 
 \rho \sum_{l=1}^k \alpha^k_l d_{l-1} \bq^{l-2} 
 = \rho P_k \sum_{l=1}^k \alpha^k_l d_{l-1} \bq^{l-2} 
 = \rho Q_k  \left(Q_k^\top Q_k\right)^{\dag} 
   \sum_{l=1}^k \alpha^k_l d_{l-1} Q_k^\top \bq^{l-2} 
\]
(with $\bq^{-1} = 0$), 
\begin{align*} 
 \rho Q_k \left(Q_k^\top Q_k\right)^{\dag}U_k^\top P_k^\bot \bq^k &= 
  Q_k \left(Q_k^\top Q_k\right)^{\dag} 
    \left(U_k^\top \bq^k - U_k^\top P_k \bq^k\right) \\
   &= \rho Q_k  \left(Q_k^\top Q_k\right)^{\dag} 
  \left(U_k^\top \bs{q}^k - U_k^\top Q_k \balpha^k\right) \\
   &= \rho Q_k  \left(Q_k^\top Q_k\right)^{\dag}
   \left(U_k^\top \bq^k - \sum_{l=1}^k \alpha_l^k U_k^\top \bq^{l-1} \right), 
\end{align*}
and 
\[
- \rho d_k P_k \bq^{k-1} = 
  - \rho d_k Q_k  \left(Q_k^\top Q_k\right)^{\dag} Q_k^\top \bs{q}^{k-1}. 
\]
Injecting these equations into~\eqref{uk+1}, we obtain 
\[
\bu^{k+1} = \sum_{l=1}^k \alpha^k_l \bu^l 
 + \rho Q_k \left(Q_k^\top Q_k\right)^{\dag} \left( 
 U_k^\top \bq^k - d_k Q_k^\top \bs{q}^{k-1} 
- \sum_{l=1}^k \alpha_l^k 
  \left( U_k^\top \bq^{l-1} - d_{l-1} Q_k^\top \bq^{l-2} \right) \right) 
  + \cI_k(A), 
\]
which is the equation provided in the statement of the proposition. 

\end{appendix}

 \end{document}